\documentclass[12pt]{amsart}
\usepackage{amssymb,bbold,hyperref}
\usepackage[all]{xy}

\theoremstyle{plain}
\newtheorem{theorem}{Theorem}[section]
\newtheorem{proposition}[theorem]{Proposition}
\newtheorem{corollary}[theorem]{Corollary}
\newtheorem{lemma}[theorem]{Lemma}

\theoremstyle{definition}
\newtheorem{remark}[theorem]{Remark}
\newtheorem{example}[theorem]{Example}
\newtheorem{definition}[theorem]{Definition}
\newtheorem{problem}[theorem]{Problem}

\newcommand{\abs}[1]{\lvert#1\rvert}
\newcommand{\norm}[1]{\lVert#1\rVert}
\newcommand{\bigabs}[1]{\bigl\lvert#1\bigr\rvert}
\newcommand{\bignorm}[1]{\bigl\lVert#1\bigr\rVert}
\newcommand{\Bigabs}[1]{\Bigl\lvert#1\Bigr\rvert}
\newcommand{\Bignorm}[1]{\Bigl\lVert#1\Bigr\rVert}
\newcommand{\term}[1]{{\textit{\textbf{#1}}}}

\renewcommand{\mid}{\::\:}

\def\one{\mathbb 1}

\begin{document}

\title[Uo-convergence and applications to Ces\`aro means]
{Uo-convergence and its applications\\ to Ces\`aro means in Banach lattices}

\author{N. Gao}
\address{School of Mathematics, Southwest Jiaotong University,
  Chengdu, Sichuan, 610000, China.}
\email{ngao@home.swjtu.edu.cn}

\author{V.~G. Troitsky}
\address{Department of Mathematical and Statistical Sciences,
         University of Alberta, Edmonton, AB, T6G\,2G1, Canada.}
\email{troitsky@ualberta.ca}

\author{F. Xanthos}
\address{Department of Mathematics, Ryerson University, 350 Victoria St.,
Toronto, ON, M5B 2K3, Canada.}
\email{foivos@ryerson.ca}

\thanks{The authors were supported by NSERC grants}
\date{\today}

\begin{abstract}
  A net $(x_\alpha)$ in a vector lattice $X$ is said to uo-converge to $x$ if $\abs{x_\alpha-x}\wedge u\xrightarrow{\rm o}0$ for every $u\ge 0$. In the first part of this paper, we study some functional-analytic aspects of uo-convergence. We prove that uo-convergence is stable under passing to and from regular sublattices. This fact leads to numerous applications presented throughout the paper. In particular, it allows us to improve several results in \cite{GaoX:14,Gao:14}. In the second part, we use uo-convergence to study convergence of Ces\`aro means in Banach lattices. In particular, we establish an intrinsic version of Koml\'os' Theorem, which extends the main results of \cite{Komlos:67,Day:10,Jimenes:11} in a uniform way. We also develop a new and unified approach to  Banach-Saks properties and Banach-Saks operators based on uo-convergence. This approach yields, in particular, short direct proofs of several results in \cite{Dodds:07,Flores:06,Flores:08}.
\end{abstract}

\maketitle

\section{Introduction}

The notion of uo-convergence is an abstraction of almost everywhere
convergence in function spaces and originally goes back to
\cite{Nakano:48}. It was later investigated in
\cite{DeMarr:64,Wickstead:77,Kaplan:97,GaoX:14,Gao:14}.
In~\cite{GaoX:14}, uo-convergence was applied in a study of abstract martingales in the framework of vector lattices. In particular,
\cite{GaoX:14} includes an extension of Doob's (sub)martingale
convergence theorems to vector lattices. In the present paper, we
further investigate uo-convergence and present several applications of
this tool. 

The structure of the paper is as follows. In Section~\ref{sec:reg}, we obtain several new results about regular sublattices and order convergence. Recall that a sublattice $Y$ in a vector lattice $X$ is \term{regular} if $\inf A$ is the same in $X$ and in $Y$ whenever $A$ is a subset of $Y$ whose infimum exists in~$Y$. We prove that in this case, the order completion $Y^\delta$ of $Y$ is also regular in $X^\delta$. We then use this to deduce that order convergences in $X$ and in $Y$ are the same for order bounded nets of~$Y$.

In Section~\ref{uo-sec}, we apply results of Section~\ref{sec:reg} to show that a sublattice $Y$ is regular in $X$ iff the uo-convergences in $X$ and $Y$ agree. In particular, the uo-convergences in $X$ and in $X^\delta$ agree. This allows us to drop the order completeness assumptions from several results of \cite{GaoX:14,Gao:14}. In particular, we show that every disjoint sequence in a vector lattice uo-converges to zero, and that if $w$ is a weak unit then  $x_\alpha\xrightarrow{\rm uo}x$ iff $\abs{x_\alpha-x}\wedge w\xrightarrow{\rm o}0$. We show that a Banach lattice has the Positive Schur Property iff every uo- and weakly null sequence is norm null. We also discuss the relationship between uo-convergence in $X$ and order convergence in the universal completion of~$X$.

In Section~\ref{sec:AL}, we go over AL-representations of vector lattices with strictly positive functionals. Recall that if $X$ is a vector lattice with a strictly positive functional~$h$, then $\norm{x}=h\bigl(\abs{x}\bigr)$ defines an AL-norm on $X$ and, therefore, the completion of $X$ with respect to this norm is lattice isometric to $L_1(\mu)$ for some measure~$\mu$.  We show that~$X$, viewed as a sublattice of $L_1(\mu)$, is regular iff it is order dense iff $h$ is order continuous. In this case, the results of Section~\ref{uo-sec} yield that a sequence uo-converges in $X$ iff it converges $\mu$-almost everywhere to some vector in~$X$.

Section~\ref{sec:komlos} is centred around the Koml\'os property. Let
$(x_n)$ be a sequence in a vector space~$X$. Consider the sequence $(a_n)$ of \term{Ces\`aro means} of $(x_n)$, defined by $a_n=\frac1n\sum_{k=1}^nx_k$. In \cite{Komlos:67}, Koml\'os proved the following celebrated result:

\begin{theorem}[\cite{Komlos:67}] \label{kom}
Let $(x_n)$ be a norm bounded sequence in $L_1(\mathbb{P})$, where $\mathbb P$ is a probability measure. Then there exists a subsequence $(y_n)$ of $(x_n)$ and a function $g\in L_1(\mathbb{P})$ such that the Ces\`aro means of any subsequence of $(y_n)$ converge to $g$ almost everywhere.
\end{theorem}

We introduce the notions of Koml\'os and pre-Koml\'os properties for Banach lattices in terms of uo-convergence. Our definitions are measure-free, yet are shown to be consistent with the measure-dependent definitions given in \cite{Jimenes:11,Day:10}. In Theorem~\ref{pre-komlos}, we identify a large class of Banach lattices that possess the pre-Koml\'os and Koml\'os properties. We also study the converse of the Koml\'os theorem (Theorem~\ref{con-kom}). As will be illustrated, our results unify and improve the main results in \cite{Jimenes:11,Day:10}.

In Section~\ref{sec:BS}, we use the pre-Koml\'os property of Banach lattices to study Banach-Saks properties and Banach-Saks operators. Recall the following classical fact due to Banach-Saks \cite{Banach:30} and Szlenk \cite{Szlenk:65}.

\begin{theorem}[Banach-Saks-Szlenk]\label{BS}
Let $(x_n)$ be a weakly null sequence in $L_p(\mathbb{P})$, where $\mathbb P$ is a probability measure and $1\leq p<\infty$. Then there exists a subsequence $(y_n)$ of $(x_n)$ such that the Ces\`aro means of any subsequence of $(y_n)$ converge to zero in norm.
\end{theorem}


A Banach space is said to have the \term{(weak) Banach-Saks property} if every bounded (respectively, weakly null) sequence has a subsequence whose Ces\`aro means converge in norm. We study these properties and their ``disjoint'' variants in Banach lattices. We show, in particular, that Banach lattices with the Positive Schur Property have the weak Banach-Saks property. This immediately implies Theorem 5.7(i) in \cite{Dodds:04} that every separable Lorentz space has the weak Banach-Saks property.

Uo-convergence also provides a new and efficient way of handling domination problems of (weakly) Banach-Saks operators. We use it to develop short proofs of some of the results of \cite{Flores:06,Flores:08}, as well as of some new domination results for weakly Banach-Saks operators. We also present a variant of Kade\v c-Pe{\l}czy\'nski dichotomy in terms of uo-convergence.

\section{Order convergence and Regular sublattices}
\label{sec:reg}

Throughout this paper, $X$ stands for a vector lattice. We refer to \cite{Aliprantis:06,Aliprantis:03,Abramovich:02} for unexplained terminology on vector and Banach lattices. All vector lattices are assumed to be Archimedean. We recall a few standard definitions.

\begin{definition}\label{def:oconv}
  A net $(x_\alpha)_{\alpha\in\Gamma}$ in a vector lattice $X$ is said to \term{converge in order} to $x\in X$, written as $x_\alpha\xrightarrow{\rm o}x$, if there exists another net $(a_\gamma)_{\gamma\in \Lambda}$ in $X$ satisfying $a_\gamma\downarrow 0$ and for any $\gamma\in\Lambda$ there exists $\alpha_0\in \Gamma$ such that $\abs{x_\alpha-x}\leq a_\gamma$ for all $\alpha\geq \alpha_0$. We say that a net $(x_\alpha)$ is \term{order Cauchy} if the double net $(x_\alpha-x_\beta)_{(\alpha,\beta)}$ converges in order to zero.
\end{definition}

\begin{remark}\label{ocompl-oconv}
It is easy to see that for an order bounded net $(x_\alpha)$ in an order complete vector lattice,
\begin{displaymath}
  x_\alpha\xrightarrow{\rm o}x
  \quad\mbox{iff}\quad
  \inf_\alpha\sup_{\beta\ge\alpha}\abs{x_\beta-x}=0
  \quad\mbox{iff}\quad
  x=\inf_\alpha\sup_{\beta\ge\alpha}x_\beta=\sup_\alpha\inf_{\beta\ge\alpha}x_\beta.
\end{displaymath}
It follows that the dominating net $(a_\gamma)$ in Definition~\ref{def:oconv} may be chosen over the same index set as the original net. In case of a $\sigma$-order complete vector lattice, the same holds for sequences.
\end{remark}

The following fact is standard. It follows easily from the double equality in Remark~\ref{ocompl-oconv}; order boundedness is obtained by passing to a tail.

\begin{proposition}\label{oconv-oCauchy}
  Every order Cauchy net in an order complete vector lattice is order convergent. Every order Cauchy sequence in a $\sigma$-order complete vector lattice is order convergent.
\end{proposition}

%


\begin{definition}
  A sublattice $Y$ of a vector lattice $X$ is said to be
\begin{itemize}
\item  \term{order dense} if for every $0<x\in X$ there exists $0<y\in Y$ such that $y\le x$;
\item \term{dense with respect to order convergence} if every vector in $X$ is the order limit of a net in~$Y$;
\item \term{majorizing} if for every $0<x\in X$ there exists $y\in Y$ such that $x\le y$;
\item \term{regular} if for every subset $A$ of~$Y$, $\inf A$ is the same in $X$ and in $Y$ whenever $\inf A$ exists in~$Y$.
\end{itemize}
\end{definition}

The following fact is straightforward; see, e.g.,~\cite[Theorem~1.20]{Aliprantis:03}.

\begin{lemma}\label{regular}
  Let $Y$ be a sublattice of~$X$. The following are equivalent.
  \begin{enumerate}
  \item $Y$ is regular;
  \item If $\sup A$ exists in $Y$ then $\sup A$ exists in $X$ and the
    two suprema are equal;
  \item $y_\alpha\xrightarrow{\rm o}y$ in $Y$ implies $y_\alpha\xrightarrow{\rm o}y$ in~$X$;
  \item $y_\alpha \downarrow 0$ in $Y$ implies $y_\alpha \downarrow 0$ in~$X$;
  \end{enumerate}
\end{lemma}

It is easy to see that every ideal is regular. Furthermore, order dense sublattices are regular by \cite[Theorem~1.23]{Aliprantis:03}. It is shown in \cite[Theorem~1.27]{Aliprantis:03} that a sublattice $Y$ is order dense in $X$ iff  $a=\sup[0,a]\cap Y$ for every $a\in X_+$,
where the $\sup$ is evaluated in~$X$. Therefore, if $Y$ is order dense in $X$ then $Y$ is dense with respect to order convergence. The converse fails in general.

\begin{example}
  Let $X$ be the set of real-valued functions on $[0,1]$ of form
  $f=g+h$ where $g$ is continuous and $h$ vanishes except at finitely
  many points. Being a sublattice of $\mathbb R^{[0,1]}$, $X$ is a
  vector lattice. Let $Y=C[0,1]$. Clearly, $Y$ is a sublattice
  of~$X$. It is easy to see that $Y$ is dense with respect to order
  convergence (even sequentially), but $Y$ is not order dense in $X$:
  there is no $g\in Y$ with $0<g\le \chi_{\{\frac12\}}$. Observe also
  that $Y$ is not regular in~$X$. Indeed, let $(f_n)$ be a decreasing
  sequence in $Y_+$ such that $f_n(\frac12)=1$ for every $n$ and
  $f_n(t)\to 0$ for every $t\ne\frac12$. Then $f_n\downarrow 0$ in $Y$
  but $f_n\downarrow\chi_{\{\frac12\}}$ in~$X$.
\end{example}

\begin{lemma}\label{od-maj}
  For a sublattice $Y$ in a vector lattice~$X$, the following are equivalent.
  \begin{enumerate}
  \item\label{od-maj-maj} $Y$ is both order dense and majorizing in~$X$;
  \item\label{od-maj-inf} For every $x\in X$ one has $x=\inf\{y\in Y\mid
    y\ge x\}$;
  \item\label{od-maj-sup} For every $x\in X$ one has $x=\sup\{y\in Y\mid y\le x\}$.
  \end{enumerate}
\end{lemma}

\begin{proof}
  It is straightforward (replacing $x$ with $-x$) that
  \eqref{od-maj-inf}$\Leftrightarrow$\eqref{od-maj-sup}.  To show that
  \eqref{od-maj-maj}$\Rightarrow$\eqref{od-maj-inf}, put $A=\{y\in
  Y\mid y\ge x\}$. Since $Y$ is majorizing, there exists $y_0\in Y$
  such that $x\le y_0$. In particular, $A$ is non-empty and
  $[x,y_0]\cap Y\subseteq A$. Then
  \begin{displaymath}
    \inf[x,y_0]\cap Y
    =y_0-\sup[0,y_0-x]\cap Y
    =y_0-(y_0-x)=x,
  \end{displaymath}
  hence $\inf A=x$.

  It is left to deduce \eqref{od-maj-maj} from the other two
  statements. First, note that~\eqref{od-maj-inf} implies that $Y$
  is majorizing. Fix $x\in X_+$. By \eqref{od-maj-sup}, $x=\sup B$
  where $B=\{y\in Y\mid y\le x\}$. On the other hand, for every $y\in
  B$ we have $y\le y^+\in [0,x]\cap Y$. It follows that $x=\sup
  [0,x]\cap Y$.
\end{proof}

\medskip

Note that even when $Y$ is a regular sublattice of~$X$, order convergence in $X$ generally does not imply order convergence in~$Y$. For example, $c_0$ is a regular sublattice of $\ell_\infty$, $e_n\xrightarrow{\rm o}0$ in $\ell_\infty$ but not in~$c_0$. We will see, however, that order convergence in $X$ does imply order convergence in $Y$ under certain additional assumptions. The following theorem is essentially in~\cite{Abramovich:05}. We provide a proof for the convenience of the reader.

\begin{theorem}\label{odense-maj-o-lim}
  Suppose that $Y$ is order dense and majorizing.  Then
  $x_\alpha\xrightarrow{\rm o}0$ in $Y$ iff $x_\alpha\xrightarrow{\rm o}0$ in
  $X$ for any net $(x_\alpha)$ in~$Y$.
\end{theorem}

\begin{proof}
  Since $Y$ is regular, the forward implication is obvious from
  Lemma~\ref{regular}. Suppose now that
  $x_\alpha\xrightarrow{\rm o}0$ in~$X$. Let $(a_\gamma)$ be a net in $X$
  as in Definition~\ref{def:oconv}.  Put
  \begin{displaymath}
    A=\bigl\{y\in Y\mid y\ge a_\gamma\text{ for some }\gamma\bigr\}.
  \end{displaymath}
  Then $\inf A=0$ in $X$ and, therefore, in~$Y$. Indeed, if $z\in X$ and $0\le z\le A$ then for
  every $\gamma$ we have $z\le\{y\in Y\mid y\ge a_\gamma\}$, so that
  $z\le a_\gamma$ by Lemma~\ref{od-maj}. Hence, $z=0$.

  Since $A$ is directed downwards, we may view $A$ as a decreasing net
  in~$Y$. It is easy to see that this net dominates $(x_\alpha)$ in
  the sense of Definition~\ref{def:oconv}.
\end{proof}

For a vector lattice~$X$, we write $X^\delta$ for its order (or Dedekind) completion. Recall from \cite[Theorem~1.41]{Aliprantis:03} that $X^\delta$ is the unique (up to a lattice isomorphism) order complete vector lattice that contains $X$ as a majorizing and order dense sublattice. In particular, $X$ is a regular sublattice of~$X^\delta$.

\begin{corollary}[\cite{Abramovich:05}]\label{AS}
  For every net $(x_\alpha)$ in~$X$, $x_\alpha\xrightarrow{\rm o}0$ in $X$ iff $x_\alpha\xrightarrow{\rm o}0$ in~$X^\delta$.
\end{corollary}

\begin{theorem}\label{reg_prop1}
  Let $Y$ be a regular sublattice of a vector lattice~$X$. Then
  $Y^\delta$ is a regular sublattice of~$X^\delta$.
\end{theorem}

\begin{proof}
  Since $X$ is regular in~$X^\delta$, we conclude that $Y$ is regular in~$X^\delta$. Thus, without loss of generality, we may assume that $X$ is order complete. Let $J\colon Y\to X$ be the inclusion mapping. Then $J$ is order continuous by regularity of~$Y$. By \cite[Theorem~1.65]{Aliprantis:06}, the operator $J$ can extended to an order continuous positive operator $T\colon Y^\delta\to X$. We will show that $T$ is a lattice isomorphism from $Y^\delta$ into~$X$.

Pick any $a\in Y^\delta$. Take two nets $(y_\alpha)$ and $(z_\alpha)$ in $Y$ such that $0\leq y_\alpha\uparrow a^+$ and $0\leq z_\alpha\uparrow a^-$ in~$Y^\delta$. Then it is clear that
$$y_\alpha=Ty_\alpha\xrightarrow{\rm o}T(a^+)\mbox{ in }X.$$
Moreover, since $y_\alpha-z_\alpha\xrightarrow{\rm o} a$ in~$Y^\delta$, we have $y_\alpha-z_\alpha=T(y_\alpha-z_\alpha)\xrightarrow{\rm o}Ta$ in~$X$. Also, since $a^+\wedge a^-=0$, we have $y_\alpha\wedge z_\alpha=0$ in $Y^\delta$ for any~$\alpha$, and hence
$$y_\alpha=(y_\alpha-z_\alpha)^+\xrightarrow{\rm o}(Ta)^+\mbox{ in }X.$$
Therefore, $T(a^+)=(Ta)^+$ for any $a\in Y^\delta$. It follows that $T$ is a lattice homomorphism.

Suppose now that $Ta=0$ for some $a\in Y^\delta$. Since $T$ is a lattice homomorphism, we may assume that $a\geq 0$. Take $(y_\alpha)$ in $Y$ such that $0\leq y_\alpha\uparrow a$ in~$Y^\delta$. Then $0\leq y_\alpha=Ty_\alpha\leq  Ta=0$, implying $y_\alpha=0$ for all~$\alpha$. Hence, $a=0$. This proves that $T$ is one-to-one.

The regularity of $Y^\delta$ in $X$ follows from the order continuity of~$T$.
\end{proof}

\begin{lemma}\label{ocompl-reg}
  Let $Y$ be a regular order complete sublattice of~$X$. Suppose that $y_\alpha\xrightarrow{\rm o}x$ in $X$ for some order bounded net $(y_\alpha)$ in $Y$ and some vector $x\in X$. Then $x\in Y$ and $y_\alpha\xrightarrow{\rm o}x$ in~$Y$.
\end{lemma}

\begin{proof}
Replacing $X$ with~$X^\delta$, we may assume that $X$ is order complete. By Remark~\ref{ocompl-oconv},  $x=\inf_\alpha\sup_{\beta\ge\alpha}y_\beta=\sup_\alpha\inf_{\beta\ge\alpha}y_\beta$, where the $\sup$ and the $\inf$ are evaluated in~$X$. Since $Y$ is order complete, the the $\sup$ and the $\inf$ exist in~$Y$; they have the same values as in $X$ because $Y$ is regular in~$X$. It follows that $x\in Y$ and $y_\alpha\xrightarrow{\rm o}x$ in~$Y$.
\end{proof}

\begin{corollary}\label{reg-obdd-twoway}
  If $Y$ is a regular sublattice of $X$ then $x_\alpha\xrightarrow{\rm o}0$ in $Y$ iff $x_\alpha\xrightarrow{\rm o}0$ in $X$ for every order bounded net $(x_\alpha)$ in~$Y$.
\end{corollary}

\begin{proof}
If $x_\alpha\xrightarrow{\rm o}0$ in $Y$ then $x_\alpha\xrightarrow{\rm o}0$ in $X$ by Lemma~\ref{regular}. For the converse implication, suppose that $(x_\alpha)$ is an order bounded net in $Y$ and  $x_\alpha\xrightarrow{\rm o}0$ in~$X$. By Corollary~\ref{AS} and Theorem~\ref{reg_prop1}, we may assume that $X$ and $Y$ are both order complete. Now apply Lemma~\ref{ocompl-reg}.
\end{proof}

\begin{corollary}\label{reg_prop2}
Let $Y$ be a regular sublattice of~$X$. If $Y$ is dense in $X$ with respect to order convergence then $Y$ is order dense in~$X$. If, in addition, $Y$ is order complete in its own right, then $Y$ is an ideal of~$X$.
\end{corollary}

\begin{proof}
By Theorem~\ref{reg_prop1}, $Y^\delta$ is a regular sublattice of~$X^\delta$. We first show that $Y^\delta$ is an ideal of~$X^\delta$. Let $b \in Y^\delta$ and $a \in X^\delta$ be such that $0 \leq a \leq b$. Denote
\begin{math}
  [0,a]=\bigl\{z \in X^\delta\mid  0\leq z \leq a\bigr\}.
\end{math}
By order denseness of $X$ in~$X^\delta$, we have
$a=\sup[0,a] \cap X$ in~$X^\delta$.
Let $x \in [0,a]\cap X$. Since $Y$ is dense in $X$ with respect to order convergence, there exists a net $(y_\alpha)$ in $Y$ such that $y_\alpha \xrightarrow{\rm {o}} x$ in $X$ and, therefore, in~$X^\delta$. Put $z_\alpha=\abs{y_\alpha}\wedge b$. Then $(z_\alpha)$ is an order bounded net in~$Y^\delta$, and $z_\alpha \xrightarrow{\rm o}x$ in~$X^\delta$. Lemma~\ref{ocompl-reg} yields $x \in Y^\delta$. Therefore, $a=\sup[0,a] \cap Y^\delta$ in~$X^\delta$. But again, since $[0,a]\cap Y^\delta$ is order bounded in~$Y^\delta$, its supremum in $X^\delta$ equals to its supremum in~$Y^\delta$. Hence, $a \in Y^\delta$. This proves that $Y^\delta$ is an ideal of~$X^\delta$.

Now pick any $x \in X_+$ with $x>0$. Since $Y$ is dense with respect to order convergence in~$X$, there exists a net $(y_\alpha)$ in $Y$ such that $y_\alpha \xrightarrow{\rm o} x$ in~$X$. Then $\abs{y_\alpha}\wedge x \xrightarrow{\rm o} x$ in~$X$. It follows that $z:=\abs{y_{\alpha_0}}\wedge x >0$ for some $\alpha_0$. Since $Y^\delta$ is an ideal, it follows that $z\in Y^\delta$. Using order denseness of $Y$ in~$Y^\delta$, we can find $y \in Y$ such that $0<y\leq z\leq x$. This proves that $Y$ is order dense in~$X$.

Finally, note that if $Y$ is order complete, then $Y=Y^\delta$ is an ideal of $X^\delta$ and hence of~$X$.
\end{proof}

The ``in addition'' part also follows from the standard fact that an order complete order dense sublattice is an ideal; \cite[Theorem~2.31]{Aliprantis:06}.

\section{Unbounded order convergence and regular sublattices}\label{uo-sec}

Following \cite{Nakano:48,DeMarr:64,Wickstead:77,Kaplan:97,GaoX:14}, a net $(x_\alpha)$ in a vector lattice $X$ is said to \term{converge in unbounded order} (uo-converge for short) to $x\in X$, written as $x_\alpha\xrightarrow{\rm uo}x$, if $\abs{x_\alpha-x}\wedge y\xrightarrow{\rm o}0$ for any $y\in X_+$;
$(x_\alpha)$ is said to be \term{uo-Cauchy} if the ``double'' net $(x_\alpha-x_\beta)_{(\alpha,\beta)}$ uo-converges to zero. It is easily seen that uo-convergence (respectively, uo-Cauchy) coincides with order convergence (respectively, o-Cauchy) for order bounded nets. But in general, they are very different; for example, the sequence $(e_n)$ of the standard unit vectors in $c_0$ uo-converges (to zero), but does not converge in order. We refer to \cite{GaoX:14,Gao:14} for some basic properties of uo-convergence and uo-Cauchy.

Throughout this paper, measures and vector measures are always assumed to be countably additive; no finiteness is assumed unless specified otherwise.
Given a measure space $(\Omega,\Sigma,\mu)$, we write $L_0(\mu)$ for the vector lattice of real-valued measurable functions on $\Omega$ modulo almost everywhere (a.e.) equality equipped with the a.e.\ order: $f\geq g$ iff $f(t)\geq g(t)$ for a.e.~$t\in \Omega$. It a standard fact that if a sequence of measurable functions converges a.e.\ then it converges a.e.\ to a measurable function. It follows easily that $L_0(\mu)$ is $\sigma$-order complete. 

\begin{proposition}\label{L0}
For a sequence $(x_n)$ in $L_0(\mu)$, the following are equivalent:
\begin{enumerate}
\item\label{L0-uo-conv} $(x_n)$ is uo-convergent;
\item\label{L0-uo-Cauchy} $(x_n)$ is uo-Cauchy;
\item\label{L0-ae} $(x_n)$ converges a.e.;
\item\label{L0-oconv} $(x_n)$ is order convergent;
\item\label{L0-o-Cauchy} $(x_n)$ is order Cauchy.
\end{enumerate}
In this case, $(x_n)$ is order bounded and the limits in \eqref{L0-uo-conv}, \eqref{L0-ae}, and~\eqref{L0-oconv} are the same.
\end{proposition}

\begin{proof}
  The implications \eqref{L0-uo-conv}$\Rightarrow$\eqref{L0-uo-Cauchy}, \eqref{L0-oconv}$\Rightarrow$\eqref{L0-o-Cauchy}, and \eqref{L0-oconv}$\Rightarrow$\eqref{L0-uo-conv} are trivial; Proposition~\ref{oconv-oCauchy} yields \eqref{L0-o-Cauchy}$\Rightarrow$\eqref{L0-oconv}.

\eqref{L0-uo-Cauchy}$\Rightarrow$\eqref{L0-ae}
Suppose that $(x_n)$ is uo-Cauchy in $L_0(\mu)$ but $(x_n)$ is not a.e.\ convergent. It follows that there exists an $\varepsilon>0$ such that the set
\begin{displaymath}
  A=\Bigl\{t\in\Omega\mid\inf_{n,m\geq 1} \sup_{k \geq n,l\geq m}\bigabs{x_k(t)-x_l(t)} >\varepsilon\Bigr\}
\end{displaymath}
has positive measure. Since $\abs{x_n-x_m}\wedge \chi_{A}\xrightarrow{\rm o}0$, we have
$$v_{n,m}:=\sup_{k \geq n,l\geq m}\abs{x_k-x_l}\wedge\chi_{A}\downarrow0 \text{ in } L_0(\mu).$$
But for any $n,m\geq 1$, $$v_{n,m}(t)= \sup_{k \geq n,l\geq m}\abs{x_k(t)-x_l(t)}\wedge\chi_{A}(t)\geq \varepsilon \text{ for a.e. }t\in A,$$
implying that $v_{n,m}\geq \varepsilon\chi_A>0$, a contradiction.

\eqref{L0-ae}$\Rightarrow$\eqref{L0-oconv}
Suppose $x_n\xrightarrow{\rm a.e.}x$. Without loss of generality, we may assume that $x\in L_0(\mu)$ and, replacing $x_n$ with $x_n-x$ and modifying each $x_n$ on a set of measure zero, we may assume that $x_n(t)\to 0$ for every $t\in\Omega$. 
Then $\sup_n\abs{x_n(t)}<\infty$ for every $t\in\Omega$. It follows easily that this pointwise supremum is also the supremum of $(\abs{x_n})$ in $L_0(\mu)$. Therefore, $(x_n)$ is order bounded. For every $t$ and~$n$, put $z_n(t)=\sup_{k\ge n}\bigabs{x_k(t)}$. It is easy to see that $z_n=\sup_{k\ge n}\abs{x_k}$ in $L_0(\mu)$ and that $z_k(t)\downarrow 0$ for every~$t$. It follows that $\abs{x_n}\le z_n\downarrow 0$ in $L_0(\mu)$; hence  $x_n\xrightarrow{\rm o}0$.
\end{proof}

Understanding the relations of uo-convergence in the entire vector lattice and in a sublattice is of critical importance to applications of uo-convergence; see \cite{GaoX:14,Gao:14}. In general, uo-convergence may not be stable under passing to and from sublattices. The following theorem identifies the sublattices for which uo-convergence does pass to and from them; this theorem is key to numerous applications of uo-convergence. Cf.\ Theorem~\ref{odense-maj-o-lim} and Corollary~\ref{reg-obdd-twoway}.

\begin{theorem}\label{uo_regular}
Let $Y$ be a sublattice of a vector lattice~$X$. The following are equivalent:
\begin{enumerate}
\item\label{uo_prop2i1} $Y$ is regular;
\item\label{uo_prop2i2} For any net $(y_\alpha)$ in~$Y$, $y_\alpha \xrightarrow{\rm uo} 0$ in $Y$ implies $y_\alpha \xrightarrow{\rm {uo}} 0$ in~$X$;
\item\label{uo_prop2i3} For any net $(y_\alpha)$ in~$Y$, $y_\alpha \xrightarrow{\rm uo} 0$ in $Y$ if and only if $y_\alpha \xrightarrow{\rm {uo}} 0$ in~$X$.
\end{enumerate}
\end{theorem}

\begin{proof}
The implication \eqref{uo_prop2i3}$\Rightarrow$\eqref{uo_prop2i2} is obvious. To prove that \eqref{uo_prop2i2}$\Rightarrow$\eqref{uo_prop2i1},
suppose that $y_\alpha\downarrow 0$ in~$Y$. Applying the fact that uo-convergence agrees with order convergence for order bounded nets to a tail of $(y_\alpha)$, one can easily obtain $y_\alpha\downarrow 0$ in $X$ by \eqref{uo_prop2i2}. It follows that $Y$ is regular.

To prove that \eqref{uo_prop2i1}$\Rightarrow$\eqref{uo_prop2i3}, suppose that $Y$ is regular in~$X$. Let $(y_\alpha)$ be a net in $Y$ such that $y_\alpha \xrightarrow{\rm uo} 0$ in~$Y$. Since $X$ is regular in~$X^\delta$, it follows that $Y$ is also regular in~$X^\delta$. Let $I$ be the ideal generated by $Y$ in~$X^\delta$. We claim that $y_\alpha \xrightarrow{\rm uo} 0$ in~$I$.
Indeed, fix $u\in I_+$. There exists $y\in Y_+$ such that $0\le u\le y$. By assumption, $\abs{y_\alpha}\wedge y\xrightarrow{\rm o}0$ in $Y$ and, therefore, in~$X^\delta$, because $Y$ is regular in~$X^\delta$.
Furthermore, since $I$ is regular in~$X^\delta$, we have $\abs{y_\alpha}\wedge y\xrightarrow{\rm o}0$ in $I$ by Corollary~\ref{reg-obdd-twoway}. It follows from $0\le u\le y$ that $\abs{y_\alpha}\wedge u\xrightarrow{\rm o}0$ in~$I$. Therefore,  $y_\alpha \xrightarrow{\rm uo} 0$ in~$I$. It now follows from \cite[Lemma~3.4]{GaoX:14} that $y_\alpha \xrightarrow{\rm uo}0$ in~$X^\delta$.
Finally, for any $x\in X_+$, $\abs{y_\alpha}\wedge x\xrightarrow{\rm o}0$ in~$X^\delta$, and, therefore, in $X$ by Corollary~\ref{AS}, so that $y_\alpha \xrightarrow{\rm uo} 0$ in~$X$.

Conversely, let $(y_\alpha)$ be a net in $Y$ such that $y_\alpha \xrightarrow{\rm uo} 0$ in~$X$. Fix $u\in Y_+$. Then $\abs{y_\alpha}\wedge u\xrightarrow{\rm o}0$ in~$X$. By Corollary~\ref{reg-obdd-twoway},  $\abs{y_\alpha}\wedge u\xrightarrow{\rm o}0$ in~$Y$, so that $y_\alpha \xrightarrow{\rm uo} 0$ in~$Y$.
\end{proof}

This theorem allows us to drop the order completeness assumptions in
several known results. Namely, the following three corollaries improve
 \cite[Lemmas~3.4 and 4.5]{GaoX:14}, \cite[Theorem~2.2]{Kaplan:97} and \cite[Lemma~1.1]{Gao:14}.

\begin{corollary}\label{subl}
Suppose that $Y$ is either an ideal of a vector lattice~$X$, or an order continuous norm complete sublattice of a normed lattice~$X$. Then for a net $(y_\alpha)$ in~$Y$, $y_\alpha \xrightarrow{\rm uo} 0$ in $Y$ if and only if $y_\alpha \xrightarrow{\rm uo} 0$ in~$X$.
\end{corollary}

\begin{proof}
Simply observe that $Y$ is regular in $X$ in either case.
\end{proof}

\begin{remark}\label{uo-kothe}
Let $(\Omega,\Sigma,\mu)$ be a measure space and $X$ be an ideal, or more generally, a regular sublattice, of $L_0(\mu)$. Then for a sequence $(x_n)$ in~$X$, we have $x_n\xrightarrow{\rm uo}0$ in $X$ iff $x_n\xrightarrow{\rm uo}0$ in $L_0(\mu)$, iff $x_n\xrightarrow{\rm a.e.}0$ by Proposition~\ref{L0}.
Similarly, $(x_n)$ is uo-Cauchy in $X$ iff $(x_n)$ is uo-Cauchy in $L_0(\mu)$, iff $(x_n)$ converges almost everywhere. In the latter case, $(x_n)$ is uo-convergent in $X$ iff its a.e.~limit in $L_0(\mu)$ belongs to~$X$.

In particular, this statement holds for $L_p(\mu)$ spaces, where $0<p\leq \infty$, and for K\"othe function spaces (cf.~\cite[Definition~1.b.7]{Lindenstrauss:79}). This shows that the uo-convergence may be viewed as a generalization of a.e.\ convergence.
\end{remark}

\begin{corollary}\label{uo-weaku}
Let $X$ be a vector lattice with a weak unit $x_0>0$. Then for a net $(x_\alpha)$ in~$X$, $x_\alpha\xrightarrow{\rm uo}0$ in $X$ if and only if $\abs{x_\alpha}\wedge x_0\xrightarrow{\rm o}0$ in~$X$.
\end{corollary}

\begin{proof}
Observe that $x_0$ is also a weak unit of~$X^\delta$. By Theorem~\ref{uo_regular}, $x_\alpha\xrightarrow{\rm uo}0$ in $X$ if and only if $x_\alpha\xrightarrow{\rm uo}0$ in~$X^\delta$, and thus by \cite[Theorem~2.2]{Kaplan:97}, if and only if $\abs{x_\alpha}\wedge x_0\xrightarrow{\rm o}0$ in~$X^\delta$, which is equivalent to
$\abs{x_\alpha}\wedge x_0\xrightarrow{\rm o}0$ in~$X$.
\end{proof}

\begin{corollary}\label{dis}Let $(x_n)$ be a disjoint sequence in~$X$. Then $x_n\xrightarrow{\rm uo}0$ in~$X$.
\end{corollary}
\begin{proof}Since $(x_n)$ is disjoint in~$X^\delta$, it follows from \cite[Lemma~1.1]{Gao:14} that $x_n\xrightarrow{\rm uo} 0$ in~$X^\delta$, and therefore, in $X$ by Theorem~\ref{uo_regular}.
\end{proof}

Recall that a Banach lattice $X$ has the \term{Positive Schur Property (PSP)} if $0\leq x_n\xrightarrow{\rm w}0$ implies $x_n\to 0$ (in norm).
The following theorem was proved in \cite[Theorem~3.12]{GaoX:14} for $\sigma$-order complete spaces. We use Corollary~\ref{dis} to drop the  $\sigma$-order completeness condition.

\begin{theorem}\label{PSP_char}
  A Banach lattice has the PSP if and only if
  $x_n\xrightarrow{\rm uo,w}0$ implies $x_n\to 0$ in norm for every sequence $(x_n)$.
\end{theorem}

\begin{proof}
  If $X$ has the PSP then $X$ contains no copy of~$c_0$, so that $X$ is order continuous and, therefore, order complete. The result now follows from \cite[Theorem~3.12]{GaoX:14}. Conversely, suppose that  $x_n\xrightarrow{\rm uo,w}0$ in $X$ implies $\norm{x_n}\rightarrow0$. Again, it suffices to prove that $X$ is order continuous; it will then follow from \cite[Theorem~3.12]{GaoX:14} that $X$ has the PSP. Let $(x_n)$ be an order bounded positive disjoint sequence. It is easy to see that $x_n\xrightarrow{\rm w}0$. By Corollary~\ref{dis}, $x_n\xrightarrow{\rm uo}0$. Thus, by the assumption, $x_n\to 0$ in norm. This yields that $X$ is order continuous.
\end{proof}

\begin{remark}
In \cite[Definition 5.1]{Astashkin:08}, the authors introduce the \term{Wm property} for an r.i.\ space $X$ on $[0,1]$ as follows: $X$ is said to have the Wm property if $x_n\to 0$ whenever $x_n\xrightarrow{{\rm w},\mu}0$ (i.e.,
$(x_n)$ converges to zero weakly and in measure). We claim that this property is equivalent to the PSP. Indeed, suppose that $X$ has the PSP; let $x_n\xrightarrow{{\rm w},\mu}0$. Then every subsequence of $(x_n)$ has a further subsequence which converges to zero a.e.; hence, it is uo-null by Remark~\ref{uo-kothe} and is, therefore, norm null by Theorem~\ref{PSP_char}. It follows that $x_n\to 0$, so that $X$ has the Wm property. The proof of the converse implication is similar.
\end{remark}

We now introduce a useful way of translating uo-convergence  to order convergence, which is often easier to work with. Recall from \cite[Definition~7.1]{Aliprantis:03} that a vector lattice is said to be \term{$\sigma$-laterally complete} if every disjoint sequence has a supremum. The following elegant result is mentioned as a comment in \cite{Nakano:48} and is formally proved in \cite[Theorem~3.2]{Kaplan:97}.

\begin{theorem}\cite{Nakano:48,Kaplan:97}\label{uo-oconv-ucompl}
  A sequence $(x_n)$ in a $\sigma$-order complete and $\sigma$-laterally complete vector lattice $X$ is uo-null iff it is order null.
\end{theorem}

Modifying the proof in~\cite{Kaplan:97}, we obtain the following result; cf. Proposition~\ref{L0}.

\begin{theorem}\label{uo-ucomp0}
A sequence $(x_n)$ in a $\sigma$-order complete and $\sigma$-laterally complete vector lattice $X$ is uo-Cauchy iff it is o-convergent.
\end{theorem}

\begin{proof}
Observe that only the necessity part needs proof. Let $(x_n)$ be uo-Cauchy in~$X$. It suffices to show that $(x_n)$ is order bounded because in this case, $(x_n)$ would be o-Cauchy and thus be o-convergent by Proposition~\ref{oconv-oCauchy}. In view of \cite[Lemma~3.1]{Kaplan:97}, we may assume that $X$ has a weak unit $e>0$.

For any $x\geq 0$, denote by $B_x$ the band generated by $x$ and by $P_x$ the band projection onto~$B_x$. Put $e_x=P_xe$.
\cite[Theorem~2.8]{Kaplan:97} states that for a net $(a_\alpha)$ in $X_+$ one has $a_\alpha\xrightarrow{\rm uo}0$ iff $e_{(a_\alpha-ne)^+}\xrightarrow{\rm o}0$ for every~$n$. It thus follows that
$e_{(\abs{x_m-x_n}-e)^+}\xrightarrow{\rm o}0$ as $(m,n)\to\infty$, and, therefore,
$\inf_{n,m\geq 1}\sup_{k\geq n,l\geq m}e_{(\abs{x_k-x_l}-e)^+}=0$,
or equivalently,
$\inf_{n\geq 1}\sup_{k\geq l\geq n}e_{(\abs{x_k-x_l}-e)^+}=0$, which can be reformulated as $$d_n:=\sup_{k\geq l\geq n}e_{(\abs{x_k-x_l}-e)^+}\downarrow0.$$
Put $e_1=e-d_1$ and $e_n=d_{n-1}-d_n$ for $n\geq 2$. We claim the following three properties of $B_{e_n}$'s and $P_{e_n}$'s:

(1) $B_{e_n}$'s are disjoint. Indeed, it follows from \cite[Theorem~1.49]{Aliprantis:06} that $e_n$'s are components of $e$ and are disjoint. Hence, $B_{e_n}$'s are disjoint.

(2) $\sum_{i=1}^nP_{e_i}x\uparrow x$ for any $x\in X_+$. Indeed, 
$\sum_{i=1}^nP_{e_i}x=P_{\sum_{i=1}^ne_i}x=P_{e-d_n}x\uparrow  x$ by \cite[Theorem~1.48]{Aliprantis:06}.

(3) For each~$n$, $\bigl(P_{e_n}\abs{x_m}\bigr)_{m=1}^\infty$ has an upper bound $b_n$ in $B_{e_n}$. Observe first that $P_{e-e_{x^+}}(x+e)\leq e$ for any $x\in X$.
Indeed, since $B_{x^+}=B_{e_{x^+}}$, we have $x\leq x^+=P_{e_{x^+}}(x)\leq P_{e_{x^+}}(x+e)$, and so 
\begin{displaymath}
  P_{e-e_{x^+}}(x+e)=(x+e)-P_{e_{x^+}}(x+e)\leq e.
\end{displaymath}
Now for any $m\geq n$, we have $$P_{e-d_n}\abs{x_m-x_n}\leq P_{e-e_{(\abs{x_m-x_n}-e)^+}}\abs{x_m-x_n}\leq e,$$
and thus
$$P_{e_n}\abs{x_m-x_n}=P_{e_n}P_{e-d_n}\abs{x_m-x_n}\leq P_{e_n}e=e_n.$$
Consequently,
$$P_{e_n}\abs{x_m}\leq e_n+P_{e_n}\abs{x_n}\text{ for any }m\geq n.$$
The desired result follows immediately.

Finally, since $b_n$'s are disjoint by (1), the supremum $b:=\sup_nb_n$ exists in~$X$. For any $m\geq 1$, since $\sum_{i=1}^nP_{e_i}\abs{x_m}=\bigvee_{i=1}^nP_{e_i}\abs{x_m}\leq b$ for any~$n$, we have, by (2), that $\abs{x_m}\leq b$.
\end{proof}

\begin{remark}
  The proof of Theorem~\ref{uo-oconv-ucompl} in \cite{Kaplan:97} and our proof of Theorem~\ref{uo-ucomp0} both utilize  \cite[Theorem~2.8]{Kaplan:97}, which is stated in~\cite{Kaplan:97} only for order complete vector lattices. However, it can be easily verified that its proof in~\cite{Kaplan:97} remains valid for countably indexed nets in $\sigma$-order complete vector lattices.
\end{remark}

For a vector lattice~$X$, denote by $X^u$ its \term{universal
  completion}, cf.~\cite[Definition~7.20]{Aliprantis:03}. We would
like to thank J.J.~Grobler for suggesting a variant of the following result to us.

\begin{corollary}\label{uo-ucompl}
A sequence $(x_n)$ in a vector lattice $X$ is uo-null in $X$ iff it is o-null in~$X^u$; it is uo-Cauchy in $X$ iff it is o-convergent in~$X^u$.
\end{corollary}

\begin{proof}
Observe that $X$ is order dense and thus is regular in~$X^u$. Apply Theorem~\ref{uo_regular}, and then Theorem~\ref{uo-oconv-ucompl} or Theorem~\ref{uo-ucomp0}, respectively.
\end{proof}

We present an application of Corollary~\ref{uo-ucompl} which asserts that uo-convergence is preserved under taking Cesaro means.

\begin{corollary}
\label{uo-Cesaro-uo0}
  Let $(x_n)$ be a sequence in a vector lattice~$X$. If $(x_n)$ is uo-null (respectively, uo-Cauchy) in $X$ then so are its Ces\`aro means.
\end{corollary}

\begin{proof}
By Corollary~\ref{uo-ucompl}, it suffices to prove that if $(x_n)$ is o-convergent in $X^u$ then its Ces\`aro means o-converge to the same limit in~$X^u$. This follows immediately from the lemma below.
\end{proof}

\begin{lemma}\label{o-cesaro}
Let $(x_n)$ be a sequence in a $\sigma$-order complete vector lattice~$X$. If $x_n\xrightarrow{\rm o}0$ then the Ces\`aro means of $(x_n)$ converge in order to zero.
\end{lemma}

\begin{proof}
Since $X$ is $\sigma$-order complete, by Remark~\ref{ocompl-oconv} we can find a \emph{sequence} $(u_n)$ such that $u_n\downarrow 0$ and $\abs{x_n}\le u_n$ for every~$n$. Let $k_n$ be the integer part of~$\sqrt{n}$. Then
\begin{displaymath}
 \Bigabs{\tfrac1n\sum_{i=1}^nx_i}\le\tfrac1n\sum_{i=1}^nu_i\le
\tfrac1n\sum_{i=1}^{k_n-1}u_i+\tfrac1n\sum_{i=k_n}^{n}u_i\\ \le
\frac{k_n}{n}u_1+u_{k_n}\downarrow 0
\end{displaymath}
because $\frac{k_n}{n}u_1\downarrow 0$ and $u_{k_n}\downarrow 0$.
\end{proof}

We end this section with an interesting result that will be needed later. A subset $A$ of a vector lattice $X$ is said to be \term{uo-closed} (respectively, \term{o-closed}) in~$X$, if for any net $(x_\alpha)\subset A$ and $x\in X$ with $x_\alpha\xrightarrow{\rm uo} x$ (respectively, $x_\alpha\xrightarrow{\rm o} x$) in~$X$, one has $x\in A$.

\begin{proposition}\label{uo_prop1}
Let $X$ be a vector lattice and $Y$ a sublattice of~$X$. Then $Y$ is uo-closed in $X$ if and only if it is o-closed in~$X$.
\end{proposition}

\begin{proof}The ``only if'' part is straightforward since order convergent nets are uo-convergent.

For the ``if'' part, suppose $Y$ is order closed, and let $(y_\alpha) \subseteq Y$ and $x\in X $ be such that $y_\alpha \xrightarrow{\rm {uo}} x $ in~$X$. Then $y_\alpha^\pm\xrightarrow{\rm uo}x^\pm$ in $X$ by \cite[Lemma~3.1]{GaoX:14}.
 Thus, without loss of generality, we assume $(y_\alpha)\subset Y_+$ and $x\in X_+$. Observe that
\begin{equation}
  \label{eq:domin}
  \mbox{for every }z\in X_+,\mbox{ we have }
  \abs{y_\alpha \wedge z- x\wedge z}\leq \abs{y_\alpha-x}\wedge z
  \xrightarrow{\rm {o}}0\mbox{ in }X.
\end{equation}
It now follows that, for any $y\in Y_+$, $y_\alpha \wedge y \xrightarrow{\rm {o}} x \wedge y$ in~$X$. Since $Y$ is order closed, $x \wedge y \in Y$ for any $y \in Y_+$. On the other hand, given any $0\leq z\in Y^{\mathrm{d}}$, we have $y_\alpha\wedge z=0$ for all~$\alpha$, so that~\eqref{eq:domin} yields $x\wedge z=0$. Therefore, $x\in Y^{\mathrm{dd}}$, which is the band generated by $Y$ in~$X$. It follows that there is a net $(z_\beta)$ in the ideal generated by $Y$ such that $0\le z_\beta\uparrow x$ in~$X$. Furthermore, for every $\beta$ there exists $w_\beta\in Y$ such that $z_\beta\le w_\beta$. Then $x\geq  w_\beta\wedge x\geq z_\beta\wedge x=z_\beta\uparrow x$ in~$X$, and so $w_\beta\wedge x\xrightarrow{\rm o}x$ in~$X$. Since $w_\beta\wedge x\in Y$ and $Y$ is order closed, we get $x\in Y$.
\end{proof}

\section{AL-representations}\label{sec:AL}
In general, uo-convergence is difficult to handle as it is defined via ``local'' order convergence. Difficulties occur especially when dealing with interactions of uo-convergence and a topological convergence. In \cite{GaoX:14,Gao:14}, uo-convergence was studied using AL-representations induced by certain strictly positive functionals on the space. 

Let $X$ be a vector lattice; let $x_0^*$ be a strictly positive functional on~$X$. Define $\norm{x}_L:=x_0^*(\abs{x})$ for any $x\in X$. Then $\norm{\cdot}_L$ is a norm on~$X$. Let $\widetilde{X}$ be the norm completion of $(X,\norm{\cdot}_L)$. Then $(\widetilde{X},\norm{\cdot}_L)$ is an AL-space in which $X$ sits as a norm dense sublattice.

In this section, we further discuss AL-representations. We improve some of the results in \cite[Subsection~2.2]{GaoX:14} as we now drop the order completeness condition. In particular, we show that an AL-representation preserves uo-convergence if and only if the strictly positive functional is order continuous.

\begin{theorem}\label{repre_thm}
Let $X$ be a vector lattice with a strictly positive functional~$x_0^*$.
The following four statements are equivalent.
\begin{enumerate}
\item\label{repre_thmi1} $x_0^*$ is order continuous on~$X$.
\item\label{repre_thmi2} $X$ is a regular sublattice of~$\widetilde{X}$.
\item\label{repre_thmi3} $X$ is an order dense sublattice of~$\widetilde{X}$.
\item\label{repre_thmi4} For any net $(x_\alpha)$ in~$X$, $x_\alpha \xrightarrow{\rm uo} 0$ in $X$ if and only if $x_a \xrightarrow{\rm uo} 0$ in~$\widetilde{X}$.
\end{enumerate}
If, in addition, $X$ is order complete, then (1)-(4) are equivalent to the following:
\begin{enumerate}
\item[(5)]\label{repre_thmi5} $X$ is an ideal in~$\widetilde{X}$.
\end{enumerate}
\end{theorem}

\begin{proof}
The equivalence of \eqref{repre_thmi2} and \eqref{repre_thmi4} follows immediately from Theorem~\ref{uo_regular}. Let $(x_\alpha)$ be a net in $X$ such that $x_\alpha\downarrow 0$ in~$X$. Since $\widetilde{X}$ is an AL-space and is, therefore, order continuous, it follows that $x_\alpha\downarrow 0$ in $\widetilde{X}$ if and only if $x_0^*(x_\alpha)=\norm{x_\alpha}_L\rightarrow 0$. This proves the equivalence of \eqref{repre_thmi1} and \eqref{repre_thmi2}. Observe that $X$ is norm dense in $\widetilde{X}$ and thus is dense in $\widetilde{X}$ with respect to order convergence; cf.~\cite[Lemma~3.11]{GaoX:14}. The equivalence of \eqref{repre_thmi2}, \eqref{repre_thmi3} and (5) now follows immediately from Corollary~\ref{reg_prop2}.
\end{proof}

Note that the implication \eqref{repre_thmi1}$\Rightarrow$(5) in Theorem~\ref{repre_thm} is also proved in \cite[Proposition~2.4.16]{Meyer-Nieberg:91} using Nakano's Theorem \cite[Theorem~1.4.14]{Meyer-Nieberg:91}.

\begin{remark}\label{re-re}
Let $X$ be a vector lattice with a strictly positive order continuous functional~$x_0^*$. Since $\widetilde{X}$ is an AL-space, Kakutani's Representation Theorem and Theorem~\ref{repre_thm} yield that $X$ can be identified as a regular sublattice of $L_1(\mu)$ for some measure~$\mu$. The converse is also true. Namely, a vector lattice is lattice isomorphic to a regular sublattice of some $L_1(\mu)$ iff it admits strictly positive order continuous functionals. It is also easily seen that a vector lattice is lattice isomorphic to an ideal of some $L_1(\mu)$ iff it is order complete and admits strictly positive order continuous functionals. Cf.~\cite[Subsection~2.2]{GaoX:14}.
\end{remark}

\begin{remark}\label{re-l-infty}
  Let $X$ be a vector lattice with a strictly positive order continuous functional~$x_0^*$. Suppose, in addition, that $X$ has a weak unit~$x_0$. Then $x_0$ is also a weak unit of $\widetilde{X}$ because $X$ is order dense in~$\widetilde{X}$. In this case,  Kakutani's Representation Theorem guarantees that one could choose $\mu$ to be a finite measure and $x_0$ could correspond to the constant $\one$ function. Furthermore, assume, in addition, that $X$ is order complete. By Theorem~\ref{repre_thm}, we may view $X$ as an ideal in $L_1(\mu)$. By the preceding observation, $X$ contains $\one$ and, therefore, $X$ contains $L_\infty(\mu)$, so that $L_\infty(\mu)\subseteq X\subseteq L_1(\mu)$, where both inclusions represent order dense ideals.
\end{remark}

Variants of the representation $L_\infty(\mu)\subseteq X\subseteq L_1(\mu)$  have been extensively used in literature, see, e.g., \cite[Theorem~1.b.14]{Lindenstrauss:79} and
\cite[Theorem~2.2]{Weis:82}. Contrary to what is claimed in some of the literature, the following example shows that the assumption that the functional $x_0^*$ is order continuous cannot
be omitted if one wants $X$ to contain $L_\infty(\mu)$.

\begin{example}
Let $X=\ell_\infty$. Clearly, $X$ is order complete and has a weak (even a strong) unit. Fix a free ultrafilter $\mathcal{U}$ on~$\mathbb{N}$. By \cite[Lemma~1.59(4)]{Abramovich:02}, $\lim_{\mathcal{U}}x_n$ exists for any $x=(x_n)\in X$. We denote the limit by $y^*(x)$. By \cite[Lemma~1.60]{Abramovich:02}, $y^*$ is a linear functional on~$X$. It is also easily seen from \cite[Definition~1.58]{Abramovich:02} that $y^*$ is a lattice homomorphism on~$X$; in particular, it is positive.
Put $z^*(x)=\sum_{n=1}^\infty\frac{x_n}{2^n}$, and let
$x_0^*=y^*+z^*$. Since $z^*$ is strictly positive, so is~$x_0^*$.

Let $\widetilde{X}$ be the $L_1$-representation for $X$ and~$x_0^*$.
We claim that $\widetilde{X}$ may be identified with
$L_1(\Omega,\mu)$ where $\Omega=\mathbb N\cup\{\infty\}$ and $\mu$ is
defined by $\mu\bigl(\{n\}\bigr)=2^{-n}$ for every $n\in\mathbb N$ and
$\mu\bigl(\{\infty\}\bigr)=1$. Indeed, let
\begin{displaymath}
  X_0=\Bigl\{f\in L_\infty(\mu)\mid
      f(\infty)=y^*(x)\text{ where }x=\bigl(f(n)\bigr)_{n\in\mathbb N}\Bigr\}.
\end{displaymath}
Consider the map $T\colon X\to X_0$ given by $(Tx)(n)=x_n$ and $(Tx)(\infty)=y^*(x)$. Clearly, $T$ is a linear bijection.
Since $y^*$ is a lattice homomorphism, $X_0$ is a sublattice of $L_1(\mu)$ and $T$ is a lattice isomorphism.

Note that $x_0^*\bigl(\abs{x}\bigr)=\norm{Tx}_{L_1}$ for every $x\in X$. Thus, in order to prove that $\widetilde{X}$ is lattice isometric to $L_1(\mu)$, it is left to show that $X_0$ is norm dense in $L_1(\mu)$.
Let $u\in L_1(\mu)$ be the characteristic function of
$\{\infty\}$. Then $u\in\overline{X_0}$. Indeed, for every~$n$, define $f_n\in X_0$ as follows: $f_n(k)=0$ when $k<n$, $f_n(k)=1$ when
$k\ge n$, and $f_n(\infty)=1$. Then $f_n\to u$ in $L_1(\mu)$, so that $u\in\overline{X_0}$. Now for every $f\in L_1(\mu)$, we have $f=g+\lambda u$ for some $g\in X_0$ which agrees with $f$ on $\mathbb N$ and some appropriate $\lambda\in\mathbb R$. It follows that $f\in\overline{X_0}$. Thus, $X_0$ is dense in $L_1(\mu)$. 

We have thus proved that $\widetilde{X}$ can be identified with $L_1(\mu)$; with $X$ corresponding to~$X_0$. Now note that $u\in L_\infty(\mu)$, yet $u\notin X_0$, hence $L_\infty(\mu)$ is not contained in~$X_0$.

Note also that if $0\le g\le u$ for some $g\in X_0$ then $g=0$. It follows that $X_0$ is not order dense in $L_1(\mu)$. Theorem~\ref{repre_thm} and the remarks after it do not apply here because $y^*$ (and, therefore, $x_0^*$) is not order continuous, which can easily be verified directly.
\end{example}

We will repeatedly use the following standard fact; see e.g.,  \cite[Proposition~1.b.15]{Lindenstrauss:79}; see also \cite[Theorem 3]{Moore:71}.

\begin{proposition}[\cite{Lindenstrauss:79}]\label{oc-wu-str-pos}
  Every order continuous Banach lattice with a weak unit admits a strictly positive functional.
\end{proposition}

\begin{remark}\label{uo-ae}
  Let $X$ be a Banach lattice such that every principal band in $X$
  admits a strictly positive order continuous functional. (By
  Proposition~\ref{oc-wu-str-pos}, this is satisfied for order
  continuous Banach lattices.)  Take any sequence $(x_n)$ in~$X$. Let
  $B$ be a principal band containing $(x_n)$. For example, one can
  take $B=B_{x_0}$ where
  $x_0=\sum_{n=1}^\infty\frac{\abs{x_n}}{2^n\norm{x_n}}$. By
  assumption, $B$ admits a strictly positive order continuous
  functional~$x_0^*$. Let $L_1(\mu)$ be an AL-representation for
  $(B,x_0^*)$. Since $B$ has a weak unit, we may chose $\mu$ to be a
  probability measure (by scaling $x_0^*$). Combining
  Theorem~\ref{repre_thm}\eqref{repre_thmi4} with
  Remark~\ref{uo-kothe}, we get $x_n\xrightarrow{\rm uo}0$ in $X$ iff
  $x_n\xrightarrow{\rm uo}0$ in $B$ iff $x_n\xrightarrow{\rm uo}0$ in
  $L_1(\mu)$ iff $x_n\xrightarrow{\rm a.e.}0$ in
  $L_1(\mu)$. Similarly, $(x_n)$ is uo-Cauchy in $X$
  iff $(x_n)$ converges a.e.~to some measurable function.
\end{remark}

The following proposition is an application of this technique.

\begin{proposition}\label{w-uo}
  Let $X$ be a Banach lattice such that every principal band admits a strictly positive order continuous functional.
If $0\le x_n\xrightarrow{\rm w}0$ then $x_{n_k}\xrightarrow{\rm uo}0$ for some subsequence $(x_{n_k})$.
\end{proposition}

\begin{proof}
Let $x_0^*$ and $L_1(\mu)$ be as in Remark~\ref{uo-ae}. It follows from $x_0^*(x_n)\to 0$ that $(x_n)$ converges to zero in norm in $L_1(\mu)$. Then there is a subsequence $(x_{n_k})$ such that $x_{n_k}\xrightarrow{\rm a.e.}0$. It follows that $x_{n_k}\xrightarrow{\rm uo}0$ in~$X$.
\end{proof}

The following is a special case of Corollary~\ref{uo-Cesaro-uo0}, but the proof is now much simpler.

\begin{example}\label{uo-Cesaro-uo}
Let $X$ be a Banach lattice in which every principal band admits a strictly positive order continuous functional. If $(x_n)$ is uo-null (respectively,~uo-Cauchy), then so are its Ces\`aro means. Indeed, by Remark~\ref{uo-ae}, it suffices to observe that the statement is true for a.e.\ convergence in $L_1(\mu)$.
\end{example}

\medskip

There is also a way of representing Banach lattices as $L_1$-spaces of functions which are integrable with respect to a \emph{vector measure}.
One can thus define almost everywhere convergence of sequences in the lattice with respect to the associated vector measures. We claim that this so-defined almost everywhere convergence is also equal to uo-convergence and is thus independent of the choice of vector measure.

A systematic study of vector measures on $\delta$-rings and integration over such vector measures can be found in \cite{Masani:89a,Masani:89b};
in particular, we refer to \cite[Section~2]{Masani:89a} for basic definitions and properties. Let $\mathcal{R}$ be a $\delta$-ring of subsets of $\Omega$ and $\nu\colon\mathcal{R}\to Y$ be a vector measure, where $Y$ is a real Banach space. Let $\mathcal{R}^{loc}$ be the $\sigma$-algebra of all sets $B$ such that $B\cap A\in\mathcal{R}$ for every $A\in\mathcal{R}$. The variation of $\nu$ is the countably additive measure $\abs{\nu}\colon\mathcal{R}^{loc}\to [0,\infty]$ defined by
\begin{displaymath}
  \abs{\nu}(A)=\sup\Bigl\{\sum_{i=1}^n\bignorm{\nu(A_i)}\mid(A_i)_1^n\text{ is a disjoint sequence in }\mathcal{R}\cap 2^{A}\Bigr\}.
\end{displaymath}
A $\nu$-null set is a set $A$ in $\mathcal{R}^{loc}$ such that $\abs{\nu}(A)=0$, or equivalently, $\nu(B)=0$ for any subset $B$ of $A$ that is contained in~$\mathcal{R}$.

Let $L_0(\nu)$ be the vector lattice of all $\mathcal{R}^{loc}$-measurable real functions (modulo $\nu$-a.e.~equality), endowed with the order: $f\geq g$ iff $f(t)\geq g(t)$ except on a $\nu$-null set. That is, $L_0(\nu)=L_0\bigl(\abs{\nu}\bigr)$. Let $L_1^w(\nu)$ be the Banach lattice of all $f$ in $L_0(\nu)$ such that
\begin{displaymath}
  \norm{f}:=\sup_{y^*\in B_{Y^*}}\int \abs{f}\,d\bigabs{y^*\nu}<\infty;  
\end{displaymath}
here $\abs{y^*\nu}\colon\mathcal{R}^{loc}\to[0,\infty)$ is the variation of $y^*\nu\colon\mathcal{R}\to\mathbb{R}$.
Given $f\in L_1^w(\nu)$, we say that $f$ is $\nu$-integrable and write $f\in L_1(\nu)$ if for every $A\in\mathcal{R}^{loc}$ there exists a vector in~$Y$, denoted $\int_Af\,d\nu$, such that
\begin{displaymath}
  y^*\Bigl(\int_Af\,d\nu\Bigr)=\int_Af\,dy^*\nu
  \quad\text{for all }y^*\in Y^*.
\end{displaymath}
Theorem~2.1.2 in~\cite{Juan:11} asserts that $L_1(\nu)$ is the order continuous part of $L_1^w(\nu)$.
We refer to \cite{Juan:11} for basic properties of $L_1(\nu)$ and $L_1^w(\nu)$.

It was proved in \cite{Curbera:90,Curbera:92} that an order continuous Banach lattice $X$ with a weak unit is lattice isometric to $L_1(\nu)$ for a vector measure $\nu$ defined on a $\sigma$-algebra. It was later extended to spaces without a weak unit in~\cite{Delgado:12,Juan:11}; in this case, one has to consider a vector measure defined on a $\delta$-ring instead of a $\sigma$-algebra. Namely, a Banach lattice $X$ is order continuous iff it is lattice isometric to $L_1(\nu)$ for a vector measure $\nu$ on a $\delta$-ring.
Given such a space $X$ represented as $L_1(\nu)$ and a sequence $(x_n)$ in~$X$, we say that the sequence $\nu$-almost everywhere converges to a function $x$ if $x_n(t)\to x(t)$ except on a $\nu$-null set. Note that $x$ need not be an element of~$X$.

Our aforementioned claim is verified by the following proposition and the subsequent paragraph.

\begin{proposition}\label{uo_and_ae}
Let $Y$ be a Banach space, $\mathcal{R}$ be a $\delta$-ring of sets of $\Omega$ and $\nu\colon\mathcal{R}\to Y$ be a vector measure. Let $X$ be a regular sublattice of $L_0(\nu)$ and $(x_n)$ be a sequence in~$X$. Then $x_n \xrightarrow{\nu-\rm{a.e.}} 0$ iff $x_n \xrightarrow{\rm uo} 0$ in~$X$, and $(x_n )$ converges $\nu$-a.e. if and only if $(x_n) $ is uo-Cauchy in~$X$.
\end{proposition}

\begin{proof}
In view of Theorem~\ref{uo_regular}, we may assume $X=L_0(\nu)$. As a vector lattice, $L_0(\nu)$ is nothing but $L_0(\abs{\nu})$. Now apply Proposition~\ref{L0}.
\end{proof}

This proposition applies to $L_1^w(\nu)$ and $L_1(\nu)$ because they both are ideals of $L_0(\nu)$. This result also shows that the concept of a.e.\ convergence in $X$ is independent of a specific representation of $X$ as $L_1(\nu)$.

\section{Koml\'os properties}\label{sec:komlos}

The property described in Koml\'os' Theorem~\ref{kom} has been extensively studied by various authors (see, for example, \cite{Aldous:76,Bozorgnia:79,Halevy:79,Garling:79,Lennard:93,Cembranos:94,Day:10,Jimenes:11}), mainly due to its numerous applications in many areas of mathematics, including probability theory, function theory, and mathematical economics. In this section, we study the property for general Banach lattices. We show that this property may be extended from $L_1(\mathbb P)$ to a very large class of Banach lattices; in particular, to order continuous Banach lattices. The version (Theorem~\ref{pre-komlos}) we establish here is intrinsic and measure-free due to the use of uo-convergence. As will be seen, it also covers and unifies the main results of \cite{Day:10,Jimenes:11}.

\begin{definition}
A Banach lattice $X$ is said to have the \term{Koml\'os property} if for every norm bounded sequence $(x_n)$ in $X$ there exists a subsequence $(y_n)$ and a vector $y$ in $X$ such that the Ces\`aro means of every subsequence of $(y_n)$ uo-converge to~$y$. More generally, we say that $X$ has the \term{pre-Koml\'os property} if every norm bounded sequence $(x_n)$ in $X$ admits a subsequence $(y_n)$ such that the Ces\`aro means of any subsequence of $(y_n)$ are uo-Cauchy in~$X$.
\end{definition}

\begin{example}\label{unkom}
\emph{A Banach lattice which fails the pre-Koml\'os property.}
Let $X=\ell_\infty(\Gamma)$, where $\Gamma$ is the collection of all sequences in~$\mathbb{N}$. Applying Remark~\ref{uo-kothe} to the counting measure on~$\Gamma$, we see that a sequence in $\ell_\infty(\Gamma)$ is uo-Cauchy if and only if it is convergent coordinatewise. Take a sequence $(a_k)$ in $[-1,1]$ such that its Ces\`aro means are divergent in~$\mathbb{R}$. For any~$n$, define $x_n\in \ell_\infty(\Gamma)$ as follows: given $\gamma=(n_k)\in \Gamma$, put $x_n(\gamma)=0$ if $n\not\in\gamma$ and $x_n(\gamma)=a_k$ if $n=n_k\in \gamma$. Now for any subsequence $(x_{n_k})$ of $(x_n)$, the Ces\`aro means of $(x_{n_k})$ at the coordinate $\gamma=(n_k)$ are the same as the Ces\`aro means of $(a_k)$, and hence diverge in~$\mathbb{R}$. It follows that the sequence of the Ces\`aro means of $(x_{n_k})$ is not uo-Cauchy in $\ell_\infty(\Gamma)$.
\end{example}

\begin{example}\label{CK-Komlos}
  \emph{$C[0,1]$ fails the Koml\'os property.} Indeed, for each~$n$, define $f_n\in C[0,1]$ so that $f_n$ equals one on $[0,\frac12]$, vanishes on $[\frac12+\frac1n,1]$, and is linear on $[\frac12,\frac12+\frac1n]$. It is easy to see that Ces\`aro means of every subsequence of $(f_n)$ decrease and converge pointwise to the characteristic function of $[0,\frac12]$ and, therefore, neither converge in order nor uo-converge. Note that, however, the Ces\`aro means of every subsequence of $(f_n)$ are uo-Cauchy.
\end{example}

We do not know whether $C[0,1]$ has the  pre-Koml\'os property, or more generally, when $C(K)$ has the pre-Koml\'os property.

The Koml\'os property clearly implies the pre-Koml\'os property but the reverse implication fails in general.

\begin{example}\label{sep}
  \emph{$c_0$ has the pre-Koml\'os property but fails the Koml\'os property}. Indeed, let $(x_n)$ be a norm bounded sequence in~$c_0$. A standard diagonal process yields a subsequence $(y_n)$ of $(x_n)$ which is coordinatewise convergent. The Ces\'aro means of any subsequence of $(y_n)$ are coordinatewise convergent and hence are uo-Cauchy by Remark~\ref{uo-kothe} (applied to the counting measure on $\mathbb N$). Consequently, $c_0$ has the pre-Koml\'os property. Now let $(e_n)$ be the standard basis of~$c_0$, and put $f_n=\sum_{i=1}^ne_i$. Clearly, the Ces\`aro means of any subsequence of $(f_n)$ converge coordinatewise to $(1,1,1,\dots)$. Since uo-convergence is the same as coordinatewise convergence in~$c_0$, it follows that the Ces\`aro means of no subsequence of $(f_n)$ are uo-convergent in~$c_0$. Hence, $c_0$ fails the Koml\'os property.
\end{example}

A Banach lattice is said to be \term{boundedly uo-complete} (respectively, \term{sequentially boundedly uo-complete})
if every norm bounded uo-Cauchy net (respectively, sequence) is uo-convergent.
We will use the following two facts.

\begin{theorem}\cite[Theorem~4.7]{GaoX:14}
  \label{uo-compl}
  An order continuous Banach lattice $X$ is a
  KB-space iff it is (sequentially) boundedly uo-complete. 
\end{theorem}

\begin{theorem}\cite[Theorem~2.2]{Gao:14}
  \label{uo-compl-dual}
 The dual space of an order continuous Banach lattice is boundedly uo-complete.
\end{theorem}

\begin{proposition}
  Every (sequentially) boundedly uo-complete Banach lattice is order complete (respectively, $\sigma$-order complete).
\end{proposition}

\begin{proof}
   Let $(x_\alpha)$ be an increasing order bounded positive net in a boundedly uo-complete Banach lattice~$X$. Then $(x_\alpha)$ is order convergent in the order completion $X^\delta$ and, therefore, is order Cauchy in~$X^\delta$. It follows from Corollary~\ref{AS} that $(x_\alpha)$ is order Cauchy, and, therefore, uo-Cauchy in~$X$. By assumption, $(x_\alpha)$ uo-converges to some $x\in X$. Since $(x_\alpha)$ is order bounded, we have $x_\alpha\xrightarrow{\rm o}x$. The order limit $x$ is easily seen to be the supremum of $(x_\alpha)$ in~$X$.

  The proof of the sequential version is similar.
\end{proof}

The converse is false: $c_0$ is order complete, yet it is not sequentially boundedly uo-complete.

\begin{proposition}\label{komlos-uocom}
A Banach lattice $X$ with the Koml\'os property is sequentially boundedly uo-complete.
\end{proposition}

\begin{proof}
Let $(x_n)$ be a bounded uo-Cauchy sequence in~$X$. By Corollary~\ref{uo-ucompl}, there exists a vector $x\in X^u$ such that $x_n\xrightarrow{\rm o}x$ in~$X^u$. Again by Corollary~\ref{uo-ucompl}, it suffices to show that $x\in X$. The Koml\'os property yields a subsequence $(y_n)$ of $(x_n)$ and a vector $y\in X$ such that $\frac1n\sum_1^ny_i\xrightarrow{\rm uo}y$ in~$X$, and, therefore, $\frac1n\sum_1^ny_i\xrightarrow{\rm o}y$ in~$X^u$. Since $y_n\xrightarrow{\rm o}x$ in~$X^u$, Lemma~\ref{o-cesaro} yields $\frac1n\sum_1^ny_i\xrightarrow{\rm o}x$ in~$X^u$. It follows that $x=y\in X$.
\end{proof}

This proposition provides another reason why $C[0,1]$ fails the
Koml\'os property; cf. Example~\ref{CK-Komlos}.  Note that in
Example~\ref{unkom}, the space $\ell_\infty(\Gamma)$ is  sequentially
boundedly uo-complete by Remark~\ref{uo-kothe}; this illustrates that the converse of this proposition is false in general. We can now present a convenient criterion for determining when a Banach lattice has the Koml\'os or the pre-Koml\'os property.

\begin{theorem}\label{pre-komlos}
Let $X$ be a Banach lattice such that every principal band in $X$ admits a strictly positive order continuous functional. Then $X$ has the pre-Koml\'os property. Moreover, $X$ has the Koml\'os property iff it is sequentially boundedly uo-complete.
\end{theorem}

\begin{proof}
 Let $(x_n)$ be a norm bounded sequence in~$X$. Let $B$ and $L_1(\mu)$ be as in Remark~\ref{uo-ae}. Since $(x_n)$ is also norm bounded in $L_1(\mu)$,  Koml\'os'
Theorem~\ref{kom} yields a subsequence $(y_n)$ of $(x_n)$ such that the Ces\`aro means $(s_m)$ of \emph{any} subsequence of $(y_n)$ converge almost everywhere to some $g\in L_1(\mu)$. It follows from Remark~\ref{uo-ae} that $(s_n)$ is uo-Cauchy in~$X$. This proves that $X$ has the pre-Koml\'os property.

Suppose that $X$ is also sequentially uo-complete. By the preceding paragraph, $s_m\xrightarrow{\rm a.e.}g$ in $L_1(\mu)$ and $(s_m)$ is uo-Cauchy in~$X$. It follows that $s_m\xrightarrow{\rm uo}s$ for some $s\in X$. Proposition~\ref{uo_prop1} yields $s\in B$. By Remark~\ref{uo-ae}, $s_m\xrightarrow{\rm a.e.}s$ in $L_1(\mu)$. It follows that $s$ equals $g$ and does not depend on the choice of a subsequence. Thus, $X$ has the Koml\'os property. The other direction follows from Proposition~\ref{komlos-uocom}.
\end{proof}

\begin{corollary}\label{str-pos-pre-komlos}
  If a Banach lattice admits a strictly positive order continuous functional then it has the pre-Koml\'os property.
\end{corollary}

\begin{example}\label{ell_infty}
  \emph{$\ell_\infty$ has the Koml\'os property.} Indeed, $\ell_\infty$ admits an order continuous strictly positive functional; it is boundedly uo-complete by Theorem~\ref{uo-compl-dual}.
\end{example}

\begin{remark}\label{Moore}
  In~\cite[Theorem 3]{Moore:71}, the author identifies a large class of Banach lattices which admit strictly positive order continuous functionals. By Theorem~\ref{pre-komlos}, these spaces have the pre-Koml\'os property.
\end{remark}

\begin{proposition}\label{kom-reg}
Let $X$ be a Banach lattice which, as a vector lattice, is a regular sublattice of an order continuous Banach lattice~$Y$. Then $X$ has the pre-Koml\'os property. Moreover, $X$ has the Koml\'os property iff it is sequentially boundedly uo-complete.
\end{proposition}

\begin{proof}
Due to regularity, every principal band in $X$ is contained in a principal band of $Y$ and thus admits a strictly positive order continuous functional; cf.~Proposition~\ref{oc-wu-str-pos}. Apply Theorem \ref{pre-komlos}.
\end{proof}

The following simple characterization of the Koml\'os property for order continuous Banach lattices is an immediate consequence of this proposition and Theorem~\ref{uo-compl}. Cf. Example~\ref{sep}.

\begin{corollary}\label{kb}
Let $X$ be an order continuous Banach lattice. Then $X$ has the pre-Koml\'os property. Moreover, $X$ has the Koml\'os property iff it is a KB-space.
\end{corollary}

We observed in Example~\ref{ell_infty} that $\ell_\infty$ has the Koml\'os property; yet it is not a KB-space. Hence, the order continuity assumption in Corollary~\ref{kb} cannot be  weakened to order completeness. The following proposition shows that to verify the Koml\'os property in an order continuous Banach lattice, one does not have to consider ``further subsequences''.

\begin{corollary}
 An order continuous Banach lattice $X$ has the  Koml\'os property iff every norm bounded sequence in $X$ has a subsequence whose Ces\`aro means are uo-convergent in~$X$.
\end{corollary}

\begin{proof}
  The forward implication is trivial. Suppose $X$ fails the  Koml\'os property.
By Corollary~\ref{kb}, $X$ is not a KB-space, so that $X$ contains a lattice copy of~$c_0$. Without loss of generality, assume that $c_0\subset X$. Let $(f_n)$ be as in Example~\ref{sep}. Then by assumption, $(f_n)$ has a subsequence whose Ces\`aro means uo-converge to some $x\in X$. Note that a norm closed sublattice of an order continuous Banach lattice is order closed, and hence is uo-closed by Proposition~\ref{uo_prop1}. Therefore, $x\in c_0$. Since the Ces\`aro means uo-converge to $x$ in~$X$, they uo-converge to $x$ in $c_0$ by Corollary~\ref{subl}. This contradicts Example~\ref{sep}.
\end{proof}

\subsection{Koml\'os property in function spaces}

Variants of the Koml\'os property in function spaces have appeared in \cite{Day:10,Jimenes:11}, where the authors defined the Koml\'os property with respect to a measure in~\cite{Day:10} and a vector measure in~\cite{Jimenes:11}. We will show that many of the results in \cite{Day:10,Jimenes:11} may be viewed as special cases of our Theorem~\ref{pre-komlos} and its corollaries. Recall that Remark~\ref{uo-kothe} and Proposition~\ref{uo_and_ae} imply that, for a sequence $(x_n)$ in a regular sublattice of $L_0(\mu)$ or $L_0(\nu)$, $(x_n)$ is a.e.\ convergent iff it is uo-Cauchy; $(x_n)$ is a.e.~null iff it is uo-null. Recall also that $L_0(\nu)=L_0\bigl(\abs{\nu}\bigr)$ and $\abs{\nu}$ is a measure defined on a $\sigma$-algebra; a set is $\nu$-null if it is $\abs{\nu}$-null.
Thus, the Koml\'os properties defined in both \cite{Day:10,Jimenes:11} coincide with our notion of Koml\'os property.

Let $X$ be a Banach lattice which is a regular sublattice of $L_0(\mu)$, where $\mu$ is a measure, or of $L_0(\nu)$, where $\nu$ is a vector measure on a $\delta$-ring. Recall that $X$ is said to have the \term{weak $\sigma$-Fatou property} when for every increasing positive norm bounded sequence $(x_n)$, if $(x_n)$ converges a.e.\ to some measurable function~$x$, then $x\in X$.

\begin{proposition}\label{wsF-uo-coml}
  Let $X$ be a Banach lattice which is a regular sublattice of $L_0(\mu)$, where $\mu$ is a measure, or of $L_0(\nu)$, where $\nu$ is a vector measure on a $\delta$-ring. Then $X$ has the weak $\sigma$-Fatou property iff $X$ is sequentially boundedly uo-complete.
\end{proposition}

\begin{proof}
  Suppose that $X$ is sequentially boundedly uo-complete; let $(x_n)$ be a positive increasing norm-bounded sequence and $x_n\xrightarrow{\rm a.e.}x$ for some measurable function~$x$. Then $(x_n)$ is uo-Cauchy in $X$ and thus is uo-convergent to some $y\in X$ by assumption. Clearly, $(x_n)$ a.e.~converges to~$y$. It follows that $x=y\in X$.

Conversely, suppose that $X$ has the weak $\sigma$-Fatou property. We claim first that $X$ is $\sigma$-order complete. Indeed, let $0\leq x_n\uparrow\leq x$ in~$X$. Then $(x_n(t))$ converges a.e., so that we can put $x_0(t)=\sup_nx_n(t)$. By assumption, $x_0\in X$. It is also clear that $x_n\uparrow x_0$ in $L_0$ and thus in $X$ by Corollary~\ref{reg-obdd-twoway}. This proves the claim. Now let $(x_n)$ be a bounded uo-Cauchy sequence in~$X$. It follows from $x_n=x_n^+-x_n^-$ that we may assume that $x_n\ge 0$ for every~$n$. We have $x_n\xrightarrow{\rm a.e.}x$ for some measurable function~$x$. For each~$n$, define $y_n=\inf_{k\ge n}x_k$. Since $X$ is a regular sublattice of~$L_0$, it follows that in the definition of~$y_n$, the infimum taken in $X$ is the same as that taken in~$L_0$, or, equivalently, taken pointwise. It now follows from $x_n\xrightarrow{\rm a.e.}x$ that $y_n\xrightarrow{\rm a.e.}x$. Clearly, $y_n\uparrow$. By assumption, this yields $x\in X$.
\end{proof}

Thus, one may view sequential bounded uo-completeness as a generalization of the weak $\sigma$-Fatou property from function spaces to general Banach lattices. We can now relate our results to \cite[Theorem 1.1]{Jimenes:11}, which is stated as the main result of~\cite{Jimenes:11} and asserts that if a Banach lattice $X$ is an ideal in $L_1(\nu)$, where $\nu$ is a vector measure on a $\delta$-ring, then $X$ has the Koml\'os property iff it has the weak $\sigma$-Fatou property. In view of Proposition~\ref{wsF-uo-coml},
this result is a special case of our Propositions~\ref{kom-reg}. Moreover, our result applies not only to ideals of $L_1(\nu)$ but also to regular sublattices of $L_p(\nu)$ ($1\leq p<\infty$); simply note that $L_p(\nu)$ ($1\leq p<\infty$) is an order continuous Banach lattice and an ideal of $L_0(\nu)$ (cf.~\cite[Chapter~3]{Juan:11}).

\bigskip

Next, we will relate our results to the results of~\cite{Day:10} on the Koml\'os property in certain function spaces over measure spaces. As we mentioned earlier, the definition of the Koml\'os property in~\cite{Day:10} agrees with our definition. We introduce two large classes of function spaces which include the spaces considered in~\cite{Day:10}, and we will show that our Theorem~\ref{pre-komlos} applies to these spaces. This will imply the main results of~\cite{Day:10}.

\begin{definition}\label{wfi}
A \term{generalized K\"othe function space} over $(\Omega,\Sigma,\mu)$ is a regular sublattice of $L_0(\mu)$ endowed with a complete lattice norm $\norm{\cdot}$ such that $\Omega$ admits a countable partition $(\Omega_n)$ into measurable sets with $\int_{\Omega_n}\abs{x}\mathrm{d}\mu<\infty$ for each $n$ and each $x\in X$.
\end{definition}

This class includes
 K\"othe function spaces (see Definition~1.b.17 in \cite{Lindenstrauss:79}). However, in contrast to K\"othe function spaces, we do not require that  $\mu$ be $\sigma$-finite, or that $X$ be an ideal of $L_0(\mu)$, or that  $\chi_A$ lie in the space for every $A\in\Sigma$ with finite measure. For example, $L_1(\mu)$ is a generalized K\"othe function space for \emph{any} measure~$\mu$.

For a vector lattice~$X$, we write $X^\sim_{\rm oc}$ for the band of all order continuous functionals in $X^\sim$ (in literature, it is often denoted by $X^\sim_n$).

\begin{proposition}\label{wfi-strpos}
Every generalized K\"othe function space admits a strictly positive order continuous functional.
\end{proposition}

\begin{proof}
  In the notation of Definition~\ref{wfi}, let $\varphi_n(x)=\int_{\Omega_n}x\,d\mu$. Clearly, each $\varphi_n$ is a positive and, therefore, a bounded functional on~$X$; as well, it is order continuous on $X$ because integration is order continuous on $L_1(\Omega_n)$. The series $\varphi:=\sum_{n=1}^\infty\lambda_n\varphi_n$ converges provided that $0<\lambda_n\downarrow 0$ sufficiently rapidly. It is clear that $\varphi$ is a strictly positive functional on~$X$. Since $X^\sim_{\rm oc}$ is closed, $\varphi$ is order continuous.
\end{proof}

\begin{proposition}\label{wfi-strpos1}
Let $(\Omega,\Sigma,\mu)$ be a $\sigma$-finite measure space and $X$ be a Banach lattice which is an ideal of $L_0(\mu)$. Then $X$ admits a strictly positive order continuous functional.
\end{proposition}

\begin{proof}
Observe first that every disjoint collection of nonzero vectors in $L_0(\mu)_+$ is at most countable. Indeed, let $D$ be such a collection. Write $\Omega=\bigcup_{n=1}^\infty \Omega_n$ where $\mu(\Omega_n)<\infty$ for each~$n$. Put $D_n=\bigl\{d\in D\mid d\wedge \chi_{\Omega_n}>0\bigr\}$. 
Since $D$ is disjoint, we have, for any distinct $d_1,\dots,d_k\in D_n$, 
\begin{displaymath}
  \sum_{i=1}^k\mu\bigl(\{t\mid d_i\wedge \chi_{\Omega_n}(t)>0\}\bigr)
  =\mu\Bigl(\bigcup_{i=1}^k\bigl\{t\mid d_i\wedge \chi_{\Omega_n}(t)>0\bigr\}\Bigr)
  \leq \mu(\Omega_n)<\infty.
\end{displaymath}
Thus, in view of the fact that
$\mu\big(\{t:d\wedge \chi_{\Omega_n}(t)>0\}\big)>0$ for each $
  d\in D_n$, it follows easily that $D_n$ is at most countable. Therefore, $D=\bigcup_1^\infty D_n$ is also at most countable.

Let $\Lambda$ be a maximal disjoint collection of non-zero positive functionals in $X^\sim_{\rm oc}$. We claim that $\Lambda$ is at most countable. Indeed, for any distinct $f,g\in\Lambda$, their carriers $C_f$ and $C_g$ are disjoint bands by Nakano's Theorem \cite[Theorem~1.67]{Aliprantis:06}. For every $f\in\Lambda$, pick $0<x_f\in C_f$. Then the collection
\begin{math}
  \bigl\{x_f: f\in\Lambda\bigr\}
\end{math}
is a disjoint collection in $X$ and, therefore, in $L_0(\mu)$; hence is at most countable by the preceding claim. It follows that $\Lambda$ is at most countable. 

By Lozanovsky's Theorem \cite[Theorem 5.25]{Abramovich:02}, $X^\sim_{\rm oc}$ separates the points of~$X$. It follows that $\Lambda\ne\varnothing$. Write $\Lambda=
\{f_n\}_{n\geq 1}$. Put $f=\sum_{n\geq 1}\frac{f_n}{2^n\norm{f_n}}$. Since $X^\sim_{\rm oc}$ is closed, $f\in X^\sim_{\rm oc}$. Since $\Lambda$ is maximal, it follows that $f$ is a weak unit of $X^\sim_{\rm oc}$. It is left to show that $f$ is strictly positive. Suppose not, then $f(x_0)=0$ for some $x_0>0$, so that for any $0<g\in X^\sim_{\rm oc}$ we have $g(x_0)=\lim_n(g\wedge nf)(x_0)=0$. This contradicts the fact that $X^\sim_{\rm oc}$ separates the points of~$X$.
\end{proof}

\begin{remark}
  Under the assumptions of Proposition~\ref{wfi-strpos1}, since $X^\sim_{\rm oc}$ separates the points of~$X$, it follows easily from Nakano's Theorem \cite[Theorem~1.67]{Aliprantis:06} that an order continuous functional is strictly positive iff it is a weak unit of~$X^\sim_{\rm oc}$. Thus,  Proposition~\ref{wfi-strpos1} is essentially equivalent to \cite[Corollary 5.27]{Abramovich:02}. While the proof of \cite[Corollary 5.27]{Abramovich:02} there is very function theoretical, our proof is more direct and functional analytic in nature.
\end{remark}

The following result now follows immediately from  Theorem~\ref{pre-komlos}, Propositions~\ref{wfi-strpos} and \ref{wfi-strpos1} and Proposition~\ref{wsF-uo-coml}. 

\begin{corollary}\label{Day}
Let $X$ be either a generalized K\"othe function space over a measure space,
or a Banach lattice which is an ideal of $L_0(\mu)$ for a $\sigma$-finite measure~$\mu$. Then $X$ has the pre-Koml\'os property. Moreover, $X$ has the Koml\'os property iff it has the weak $\sigma$-Fatou property.
\end{corollary}

In~\cite[Section~2 and Definition~3.2]{Day:10}, the authors consider so called \term{finitely integrable} and \term{weakly finitely integrable} Banach function spaces. Since every positive functional on a Banach lattice is bounded, it can be easily verified that every  finitely integrable space is a K\"othe space (namely, finite integrability equals local integrability); hence it is a generalized K\"othe space. 
Furthermore, every weakly finitely integrable space is also a generalized K\"othe space. Indeed, let $\Omega_n$'s and $w_n$'s be as in \cite[Definition~3.2]{Day:10}. Put $\Omega_n^1:=\bigl\{t\in \Omega_n\mid w_n(t)\geq 1\bigr\}$,
and for $k\geq 2$, put $\Omega_n^k:=\bigl\{t\in \Omega_n\mid\frac{1}{k}\leq w_n(t)<\frac{1}{k-1}\bigr\}$.
Then $(\Omega_n^k)_{n,k}$ is a countable partition of $\Omega$ into measurable sets (adding, if necessary, a set of measure $0$), and for each $n,k\geq 1$ and each $x\in X$, $\int_{\Omega_n^k}\abs{x}\mathrm{d}\mu\leq \int_{\Omega_n}\abs{x}kw_n\,d\mu<\infty$. Note that for a weakly finitely integrable Banach function space, the underlying measure has to be $\sigma$-finite.

It is clear that Theorem~3.1 and Corollary~3.3 in \cite{Day:10} follow from either case of Corollary~\ref{Day}.

\subsection{Koml\'os sets}

We now study the converse of the Koml\'os Theorem~\ref{kom}.
The following definition is inspired by \cite{Lennard:93}.

\begin{definition}A subset $C$ of a Banach lattice $X$ is called a \term{Koml\'os set} if for every sequence $(x_n)$ in $C$ there is a subsequence $(x_{n_k})$ of $(x_n)$ and $x\in C$ such that the Ces\`aro means of any subsequence of $(x_{n_k})$ uo-converge to $x$ in~$X$.
\end{definition}

\cite[Theorem~2.2]{Lennard:93} asserts that every convex Koml\'os set in $L_1(\mu)$ is norm bounded when $\mu$ is a $\sigma$-finite measure. This interesting property was later generalized to some other Banach function spaces in \cite{Day:10}. We now recover this result for more general Banach lattices.

Recall that a vector lattice $X$ has the \term{projection property} if every band in $X$ is a projection band. It is well known that every order complete vector lattice has the projection property.

\begin{theorem}\label{con-kom}
Let $X$ be a Banach lattice with the projection property. If $X^\sim_{\rm oc}$ is a norming subspace of $X^*$ then any convex Koml\'os set $C$ in $X$ is norm bounded.
\end{theorem}

\begin{proof}
Let $C$ be a convex Koml\'os set in~$X$. Suppose first that $X$ has a weak unit. Since $X^\sim_{\rm oc}$ is a norming subspace of~$X^*$, it suffices to show that $x^*(C)$ is bounded for every $x^* \in X^\sim_{\rm oc}$. Since $X^\sim_{\rm oc}$ is a band in~$X^*$, we may assume without loss of generality that $x^*>0$; otherwise, consider $x^*_+$ and~$x^*_-$.

Let $B$ be the carrier of~$x^*$. Then $B$ is a band in~$X$; let $P$ be the corresponding band projection. Note that $B$ has a weak unit, and the restriction $x_0^*$ of $x^*$ to $B$ is a strictly positive order continuous functional on~$B$. Let $\widetilde{B}$ be the AL-representation for $(B,x_0^*)$ as in Section~\ref{sec:AL}. By Remark~\ref{re-l-infty}, $\widetilde{B}=L_1(\mu)$ for some finite measure~$\mu$.  By \cite[Lemma~3.3]{GaoX:14}, if $x_n\xrightarrow{\rm uo}x$
in $X$ then $Px_n\xrightarrow{\rm uo}Px$. It follows that $P(C)$ is a
convex Koml\'os set in~$B$. Furthermore, $P(C)$ is a Koml\'os set in
$L_1(\mu)$ by Theorem~\ref{repre_thm}\eqref{repre_thmi4} and,
therefore, $P(C)$ is norm bounded in $L_1(\mu)$ by
\cite{Lennard:93}. Observe that
\begin{displaymath}
  \bigabs{x^*(x)}\le x^*\bigl(\abs{x}\bigr)=x^*\bigl(P\abs{x}\bigr)
  =x^*\bigl(\abs{Px}\bigr)=\norm{Px}_{L_1(\mu)}
\end{displaymath}
for every $x\in C$; this yields that $x^*(C)$ is bounded.

We now consider the general case. Suppose, for the sake of contradiction, that $C$ is not norm bounded in~$X$. Pick a sequence $(x_n)$ in $C$ such that $\sup_n\norm{x_n}=\infty$. Let $B$ be the band generated by $(x_n)$ and $P$ be the corresponding band projection. Then $B$ has the projection property and a weak unit, and $B^\sim_{\rm oc}$ is a norming subspace of~$B^*$. Observe that $P(C)$ is a Koml\'os set in $B$ by \cite[Lemma~3.3]{GaoX:14} again, and is therefore norm bounded by the preceding paragraph. This leads to a contradiction since $(x_n)\subset P(C)$.
\end{proof}

\begin{corollary}
Let $X$ be an order continuous Banach lattice or a dual Banach lattice. Then every convex Koml\'os set in $X$ is norm bounded.
\end{corollary}

Recall that a vector lattice $X$ has the \term{countable sup property}, if every subset in $X$ having a supremum contains a countable subset with the same supremum. 

\begin{proposition}
Let $(\Omega,\Sigma,\mu)$ be a $\sigma$-finite measure space and $X$ be a Banach lattice which is an ideal of $L_0(\mu)$. Suppose that $\norm{x_n}\uparrow \norm{x}$ whenever $0\leq x_n\uparrow x$ in~$X$. Then $X_{\rm oc}^\sim$ is a norming subspace of~$X^*$. In particular, every convex Koml\'os set is norm bounded.
\end{proposition}

\begin{proof}
By \cite[Lemma 2.6.1]{Meyer-Nieberg:91}, $L_0(\mu)$ has the countable sup property. It follows that $X$ has the countable sup property.
Hence, from our assumption, it follows easily that $X$ satisfies the Fatou property in the sense of \cite[Definition~2.4.18]{Meyer-Nieberg:91}; namely, for any net $0\leq x_\alpha\uparrow x$ in~$X$, one has $\norm{x_\alpha}\uparrow \norm{x}$.
Recall also that $X_{\rm oc}^\sim$ separates the points of $X$
by Lozanovsky's Theorem \cite[Theorem 5.25]{Abramovich:02},
and note that $X$ is order complete. \cite[Theorem~2.4.21]{Meyer-Nieberg:91} implies that $X_{\rm oc}^\sim$ is norming. The last assertion follows from Theorem~\ref{con-kom}.
\end{proof}

This proposition includes and improves \cite[Theorems~4.1 and 4.2]{Day:10}.
Note that our Theorem~\ref{con-kom} applies to function spaces over non-$\sigma$-finite measure spaces. 
In particular, a convex Koml\'os set in $L_1(\mu)$ is norm bounded even when $\mu$ is not necessarily $\sigma$-finite.

We finish this section with two open problems.

\begin{problem}
  Let $X$ be a sequentially boundedly uo-complete Banach lattice. Does the pre-Koml\'os property imply the Koml\'os property on $X$?
\end{problem}

\begin{problem}
  Is there an unbounded convex Koml\'os set?
\end{problem}

\section{Banach-Saks properties}
\label{sec:BS}

Let $X$ be a Banach space. A sequence $(x_n)$ in $X$ is said to be \term{Ces\`aro convergent} if its Ces\`aro means converge in norm. We say that $X$ has the \term{Banach-Saks property} (\term{BSP}) if every bounded sequence has a Ces\`aro convergent subsequence. We say that $X$ has the \term{weak Banach-Saks property} (\term{WBSP}) if every weakly null sequence has a Ces\`aro convergent subsequence; in this case, it is easy to see that the Ces\`aro means of the subsequence converge to zero. Suppose now that $X$ is a Banach lattice.  We say that $X$ has the \term{disjoint Banach-Saks property} (\term{DBSP}) (respectively, \term{disjoint weak Banach-Saks property} (\term{DWBSP})) if every bounded (respectively, weakly null) disjoint sequence has a Ces\`aro convergent subsequence.

Various Banach-Saks properties have been extensively studied; see, e.g., ~\cite{Seifert:77,Dodds:04,Dodds:07,Astashkin:07,Kaminska:14,Beauzamy:79,Flores:06,Flores:08}. We will use the following classical result.

\begin{theorem}[\cite{Erdos:76}]\label{erdos}
  Every bounded sequence $(x_n)$ in a Banach space has a subsequence $(x_{n_k})$ such that either every further subsequence of $(x_{n_k})$ Ces\`aro converges to the same limit or every further subsequence of $(x_{n_k})$ Ces\`aro diverges.
\end{theorem}

In this section, we study some aspects of Banach-Saks properties for Banach lattices. The idea is to apply the pre-Koml\'os property established in Section~\ref{sec:komlos} to reduce the norm convergence of Ces\`aro means to an order property; namely, almost order boundedness. This approach also proves very efficient when dealing with domination problems of (weakly) Banach-Saks operators.

Recall that a subset $A$ of a Banach lattice $X$ is \term{almost order bounded} if for any $\varepsilon>0$ there exists $x\in X_+$
such that $A\subset [-x,x]+\varepsilon B_X$. It follows readily from the Riesz decomposition property that $A\subset [-x,x]+\varepsilon B_X$ if and only if $\sup_{a\in A}\Bignorm{\bigl(\abs{a}-x\bigr)^+}\leq\varepsilon$. Hence, if $A$ is almost order bounded, so is its convex solid hull. It is easy to see that a norm convergent sequence is almost order bounded. We will use the following fact.

\begin{theorem}[{\cite[Proposition~4.2]{GaoX:14}}]\label{uo-aobdd}
  In an order continuous Banach lattice, every almost order bounded uo-Cauchy net converges uo- and in norm to the same limit.
\end{theorem}

Combining this Theorem with Corollary~\ref{kb}, we obtain the following useful lemma.

\begin{lemma}\label{aobdd-BSP}
Let $X$ be an order continuous Banach lattice and $(x_n)$ a bounded sequence in~$X$. Suppose that every subsequence of $(x_n)$ has a further subsequence whose Ces\`aro means are almost order bounded. Then there exist a subsequence $(x_{n_k})$ of $(x_n)$ and a vector $x\in X$ such that the Ces\`aro means of any subsequence of $(x_{n_k})$ converge uo- and in norm to~$x$.
\end{lemma}

\begin{proof}
 In view of Theorem~\ref{erdos}, by passing to a subsequence, we may assume without loss of generality that either the Ces\`aro means of every subsequence of $(x_n)$ converge to the same limit (denote it by $x$), or the Ces\`aro means of every subsequence of $(x_n)$ diverge. By Corollary~\ref{kb}, passing to a further subsequence of $(x_n)$, we may assume that the Ces\`aro means of every subsequence of $(x_n)$ are uo-Cauchy. By assumption, there exists a subsequence $(x_{n_k})$ of $(x_n)$ such that the Ces\`aro means of $(x_{n_k})$ are almost order bounded, and hence converge by Theorem~\ref{uo-aobdd}. It now follows from the first sentence that  the Ces\`aro means of every subsequence of $(x_n)$ converge to~$x$.

Let $(y_n)$ be a subsequence of $(x_{n_k})$; let $(s_m)$ be the sequence of the Ces\`aro means of $(y_n)$. It now follows from $s_m\to x$ that the sequence $(s_m)$ is almost order bounded. It also follows from the first part of the proof that $(s_m)$ is uo-Cauchy. Applying Theorem~\ref{uo-aobdd} again, we conclude that $s_m\xrightarrow{\rm uo}x$.
\end{proof}

As an immediate corollary, we obtain the following characterizations of the BSP and WBSP in order continuous Banach lattices.

\begin{theorem}\label{bsps}
For an order continuous Banach lattice~$X$, the following are equivalent.
\begin{enumerate}
\item\label{bspsi1} $X$ has the BSP (respectively, the WBSP).
\item\label{bspsi2} Every bounded (respectively, weakly null) sequence has a subsequence whose Ces\`aro means are almost order bounded.
\item\label{bspsi3} For every bounded (respectively, weakly null) sequence  $(x_n)$ in~$X$, there exist a subsequence $(x_{n_k})$ of $(x_n)$ and a vector $x\in X$ such that the Ces\`aro means of any subsequence of $(x_{n_k})$ are norm and uo-convergent to~$x$.
\end{enumerate}
\end{theorem}

\begin{corollary}\label{psp_bsp}
A Banach lattice with the PSP has the WBSP.
\end{corollary}

\begin{proof}
It is known that a Banach lattice with the PSP has order continuous norm. Given a weakly null sequence, it is almost order bounded by \cite[Theorem~3.14]{GaoX:14}. So are its Ces\`aro means. Apply Theorem~\ref{bsps}.
\end{proof}

\begin{example}
  Let $X$ be a separable Lorentz space on $[0,\alpha)$ for $0<\alpha\leq\infty$. It was proved in \cite[Theorem~5.7(i)]{Dodds:04}
that $X$ has WBSP. The proof there started with the observation that every disjoint weakly null sequence in $X$ is norm null.
This clearly implies that $X$ has the DWBSP. It is then concluded in \cite{Dodds:04} that $X$ has the WBSP because of a sophisticated variant of the
subsequence splitting property established there. 

In fact, \cite[Theorem~5.7(i)]{Dodds:04} is a special case of Corollary~\ref{psp_bsp}, because the fact that  every disjoint weakly null sequence in $X$ is norm null is equivalent to the PSP by
\cite[Corollary~2.3.5]{Meyer-Nieberg:91}.
\end{example}

For the next two propositions, we need the following lemma, which is a variant of the well-known Kade\v c-Pe{\l}czy\'nski dichotomy;
cf. \cite[p.~38]{Lindenstrauss:79}.

\begin{lemma}\label{KP}
Let $X$ be an order continuous Banach lattice and $(x_n)$ a bounded sequence in~$X$. If $x_n\xrightarrow{\rm uo}0$ in~$X$, then there exist a subsequence $(x_{n_k})$ of $(x_n)$ and a disjoint sequence $(d_k)$ of $X$ such that $\norm{x_{n_k}-d_k}\rightarrow0$.
\end{lemma}
\begin{proof}
Put $x=\sum_{n=1}^\infty 2^{-n}\abs{x_n}$. Observe that $\bignorm{\abs{x_n}\wedge y}\rightarrow0$ for any $y\in X_+$. An easy induction argument yields a subsequence $(x_{n_k})$ of $(x_n)$ such that
$$\Bignorm{\abs{x_{n_{k+1}}}\wedge \Bigl(4^{{k}}\sum_{i=1}^{k}\abs{x_{n_i}}+2^{-{k}}x\Bigr)}<\frac{1}{k}.$$
for any $k\geq 1$.
Put $u_k=\bigl(\abs{x_{n_{k+1}}}-4^{{k}}\sum_{i=1}^{{k}}\abs{x_{n_i}}-2^{-{k}}x\bigr)^+$. 
Then $$\abs{x_{n_{k+1}}}-u_k=\abs{x_{n_{k+1}}}\wedge \Bigl(4^{{k}}\sum_{i=1}^{{k}}\abs{x_{n_i}}+2^{-{k}}x\Bigr)$$
and, therefore, $\Bignorm{\abs{x_{n_{k+1}}}-u_k}<\frac{1}{k}$.
By \cite[Lemma~4.35]{Aliprantis:06}, $u_k$'s are disjoint. So we are done if $(x_n)$ is a positive sequence.

For the general case, let $P_{u_{k}}$ be the band projection onto the band generated by~$u_k$. Put $d_k=P_{u_k}x_{n_{k+1}}$.
Then $d_k$'s are disjoint. Moreover,
\begin{multline*}
  \abs{x_{n_{k+1}}-d_k}=\bigabs{x_{n_{k+1}}-P_{u_k}x_{n_{k+1}}}
  =\abs{x_{n_{k+1}}}-P_{u_k}\abs{x_{n_{k+1}}}\\
  \le\abs{x_{n_{k+1}}}-\abs{x_{n_{k+1}}}\wedge u_k
  =\bigl(\abs{x_{n_{k+1}}}-u_k\bigr)^+.
\end{multline*}
It follows that
\begin{displaymath}
  \norm{x_{n_{k+1}}-d_k}\le\norm{x_{n_{k+1}}-u_k}\to 0.
\end{displaymath}
\end{proof}

This lemma allows us to replace disjoint sequences with uo-null sequences in the definition of the DBSP as follows.

\begin{proposition}\label{d-bsp}
For an order continuous Banach lattice~$X$, the following are equivalent.
\begin{enumerate}
\item\label{d-bspi1} $X$ has the DBSP,
\item\label{d-bspi11} Every bounded disjoint sequence has a subsequence whose Ces\`aro means are almost order bounded;
\item\label{d-bspi2} Every bounded uo-null sequence has a Ces\`aro convergent subsequence;
\item\label{d-bspi21} Every bounded uo-null sequence has a subsequence whose Ces\`aro means are almost order bounded.
\end{enumerate}
\end{proposition}

\begin{proof}
The equivalences \eqref{d-bspi1}$\Leftrightarrow$\eqref{d-bspi11} and \eqref{d-bspi2}$\Leftrightarrow$\eqref{d-bspi21} can be proved by applying  Lemma~\ref{aobdd-BSP}. The implication \eqref{d-bspi2}$\Rightarrow$\eqref{d-bspi1} follows from Corollary~\ref{dis}.
For \eqref{d-bspi1}$\Rightarrow$\eqref{d-bspi2}, let $(x_n)$ be a norm bounded uo-null sequence in~$X$. Then Lemma~\ref{KP} yields a subsequence $(x_{n_k})$ of $(x_n)$ and a disjoint sequence $(d_k)$ of $X$ such that $\norm{x_{n_k}-d_k}\rightarrow0$. By passing to a further subsequence, we may assume that $(d_k)$ is Ces\`aro convergent. The desired conclusion results from the following observation:
$$\Bignorm{\frac1m\sum_{i=1}^mx_{n_i}-\frac1m\sum_{i=1}^md_i}\leq \frac1m\sum_{i=1}^m\norm{x_{n_i}-d_i}\rightarrow0.$$
\end{proof}

Recall that by Corollary~\ref{uo-Cesaro-uo0}, the Ces\`aro means of any subsequences of a uo-null sequence in $X$ are also uo-null. The next result is an analogue of Proposition~\ref{d-bsp} for the DWBSP.

\begin{proposition}\label{d-wbsp}
For an order continuous Banach lattice, the following are equivalent.
\begin{enumerate}
\item\label{d-wbspi1} $X$ has the DWBSP;
\item\label{d-wbspi11} Every weakly null disjoint sequence has a subsequence whose Ces\`aro means are almost order bounded;
\item\label{d-wbspi2} Every weakly null and uo-null sequence has a Ces\`aro convergent subsequence;
\item\label{d-wbspi21} Every weakly null and uo-null sequence has a subsequence whose Ces\`aro means are almost order bounded;
\item\label{d-wbspi3} Every weakly null positive sequence has a Ces\`aro convergent subsequence;
\item\label{d-wbspi31} Every weakly null positive sequence has a subsequence whose Ces\`aro means are almost order bounded.
\end{enumerate}
\end{proposition}

\begin{proof}The equivalences \eqref{d-wbspi1}$\Leftrightarrow$\eqref{d-wbspi11}, \eqref{d-wbspi2}$\Leftrightarrow$\eqref{d-wbspi21}, and \eqref{d-wbspi3}$\Leftrightarrow$\eqref{d-wbspi31} 
follow from Lemma~\ref{aobdd-BSP}. The equivalence \eqref{d-wbspi1}$\Leftrightarrow$\eqref{d-wbspi2} can be proved in a similar fashion as in Proposition~\ref{d-bsp}. The implication \eqref{d-wbspi2}$\Rightarrow$\eqref{d-wbspi3} follows from Proposition~\ref{w-uo}. For \eqref{d-wbspi3}$\Rightarrow$\eqref{d-wbspi2}, let $(x_n)$ be a weakly null and uo-null sequence in~$X$. By \cite[Proposition~3.9]{GaoX:14}, $(\abs{x_n})$ is also weakly null. Hence, a subsequence $(\abs{x_{n_k}})$ is Ces\`aro convergent. Note that the limit must be~$0$. Finally, observe that $$\Bignorm{\frac1m\sum_{k=1}^mx_{n_k}}\leq \Bignorm{\frac1m\sum_{k=1}^m\abs{x_{n_k}}}\rightarrow0.$$
\end{proof}

\subsection{Relations between various types of Banach-Saks properties}
For a Banach lattice, the following diagram is obvious:
$$
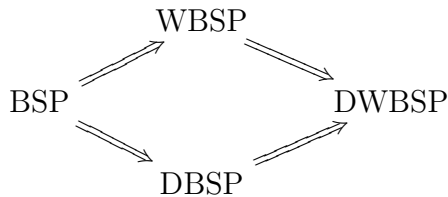
\xymatrix@R=0.5cm{
                &         \mbox{WBSP}  \ar@2{->}[dr]&    \\
 \mbox{BSP}\ar@2{->}[ur]  \ar@2{->}[dr]&    &  \mbox{DWBSP}           \\
                &        \mbox{DBSP}  \ar@2{->}[ur]    &          }
$$
We claim that, in general, none of the reverse implications hold, and the WBSP and the DBSP do not imply each other. It follows from James' Theorem 
\cite[Theorem~3.55]{Fabian:01} that Banach spaces with the BSP are reflexive (because BSP implies that every functional attains its norm on the unit ball).
Hence, Banach spaces with the BSP are just the reflexive spaces with the WBSP. Baernstein (\cite{Baernstein:72}) constructed a reflexive Banach lattice which fails the WBSP.

It is easy to see that $\ell_1$ fails the DBSP (and, therefore, the BSP); yet Corollary~\ref{psp_bsp} yields that $\ell_1$ has the WBSP (and, therefore, the DWBSP). This tells us that WBSP$\not\Rightarrow$BSP, WBSP$\not\Rightarrow$DBSP and DWBSP$\not\Rightarrow$DBSP. Being non-reflexive, $c_0$ fails the BSP. However, it is easy to see that $c_0$ has the DBSP. This yields  DBSP$\not\Rightarrow$BSP. The following example is an order continuous Banach lattice showing DBSP$\not\Rightarrow$WBSP and DWBSP$\not\Rightarrow$WBSP.

\begin{example}
Consider the space $L_p(c_0)=L_p\bigl([0,1];c_0\bigr)$, where $1<p<\infty$. By \cite[Theorem~5.1]{Dodds:04}, $L_p(c_0)$ fails the WBSP. We will show that it has the DBSP. For any $x\in L_p(c_0)$ and $\omega\in[0,1]$, write $x(\omega)=(x^1(\omega),x^2(\omega),\cdots)\in c_0$ a.e. Put $x^*(\omega)=\sup_m\abs{x^m(\omega)}$. Then $x^*\in L_p$ and $\norm{x}=\norm{x^*}_{L_p}$.

Let $(x_{i})_{i=1}^n$ be a disjoint \emph{positive} sequence in $L_p(c_0)$. We can decompose $[0,1]$ into pariwise disjoint sets $A_i$'s such that $\Omega=\bigcup_{i=1}^n A_i$ and that $x_i^* \geq x_j^*$ on $A_i$ for any $j \neq i$. Due to the disjointness of $x_i$'s, we have
$x_{1}+\dots+x_{n}=x_{1}\vee\dots\vee x_{n}$. Therefore,
\begin{multline*}
  \bigl(x_1+\dots+x_n\bigr)^*(\omega)
  =\bigl(x_1\vee\dots\vee x_n\bigr)^*(\omega)\\
  =\bigl(x_1\chi_{A_1} \vee \cdots\vee x_n\chi_{A_n}\bigr)^*(\omega)
  =\bigl(x_1\chi_{A_1} +\dots + x_n\chi_{A_n}\bigr)^*(\omega).
\end{multline*}
Since $x_{1}\chi_{A_1},\dots,x_{n}\chi_{A_n}$ have disjoint supports, we conclude that
\begin{multline*}
  \bignorm{x_{1}+\dots+x_{n}}
  =\bignorm{x_{1}\chi_{A_1} +\dots+x_{n}\chi_{A_n}}
  =\bigl(\norm{x_{1}\chi_{A_1}}^p+\dots+\norm{x_{n}\chi_{A_n}}^p\bigr)^{\frac1p}\\
  \leq\bigl(\norm{x_1}^p+\dots+\norm{x_n}^p\bigr)^{\frac{1}{p}}
  \leq n^{\frac{1}{p}}\max_i\norm{x_i} .
\end{multline*}
It follows that for any disjoint sequence $(x_{i})_{i=1}^n$ in $L_p(c_0)$, we have
\begin{displaymath}
  \bignorm{x_1+...+x_n}
  =\bignorm{\abs{x_1}+...+\abs{x_n}}
  \leq n^{\frac{1}{p}}\max_i\norm{x_i}
\end{displaymath}
In particular, $L_p(c_0)$ has the DBSP.

Observe that the above computation actually shows that $L_p(c_0)$ satisfies an upper $p$-estimate. Note also that in this example, $c_0$ may be replaced with any AM-space.
\end{example}

Recall that an order continuous Banach lattice $X$ is said to have \term{the subsequence splitting property} if for any norm bounded
sequence $(x_n)$ there exist a subsequence $(x_{n_k})$ of $(x_n)$ and two sequences $(y_k)$ and $(z_k)$ such that
$x_{n_k}=y_k+z_k$, $(y_k)$ is almost order bounded%
\footnote{In literature, in the definition of the subsequence splitting property, $(y_k)$ is required to be L-weakly compact.
However, a bounded subset of an order continuous Banach lattice is L-weakly compact if and only if it is almost order
bounded by \cite[Proposition~3.6.2]{Meyer-Nieberg:91}.}, $(z_k)$ is pairwise disjoint and $y_k\perp z_k$ for all~$k$.

\begin{remark}\label{SSP-w-null}
  Note that if the sequence $(x_n)$ in the preceding definition is weakly null, the sequences $(y_k)$ and $(z_k)$ are weakly null as well. Indeed, being almost order bounded, the sequence $(y_k)$ is relatively weakly compact. It follows that $(z_k)$ is relatively weakly compact.  Since $(z_k)$ is disjoint, it follows from \cite[Theorem~4.34]{Aliprantis:06} that both $(z_k)$ and $\bigl(\abs{z_k})$ are weakly null. It follows that $(y_k)$ is weakly null as well.
\end{remark}

The following result was obtained in \cite[Section~3]{Flores:06}. We now give an alternative proof of this result using the Koml\'os property technique. Note that this result implies Theorem~\ref{BS} because $L_p(\mu)$ has the subsequence splitting property and is easily seen to have the DWBSP (and even the DBSP when $p>1$).

\begin{proposition}[\cite{Flores:06}]\label{splitting}
Let $X$ be a Banach lattice with the subsequence splitting property.
\begin{enumerate}
\item\label{SSP-DBSP} If $X$ has the DBSP then it has the BSP.
\item\label{SSP-DWBSP} If $X$ has the DWBSP then it has the WBSP.
\end{enumerate}
\end{proposition}

\begin{proof}
We only prove \eqref{SSP-DWBSP} here; the proof of \eqref{SSP-DBSP} is similar. Let $(x_n)$ be a weakly null sequence in~$X$. Passing to a subsequence, we assume that $x_n=y_n+z_n$, where $(y_n) $ is almost order bounded, $(z_n) $ is disjoint, and both $(y_k)$ and $(z_k)$ are weakly null. Passing to a further subsequence, we may assume that every subsequence of $(y_n)$ is Ces\`aro convergent by Lemma~\ref{aobdd-BSP}. Since $X$ has the DWBSP, passing to a further subsequence, we may assume that $(z_n)$ is Ces\`aro convergent. It follows that $(x_n)$ is Ces\`aro convergent.
\end{proof}

We also have the following positive result which applies to sequence spaces.

\begin{proposition}\label{atomic-DWBSP}
For an order continuous atomic Banach lattice, the DWBSP implies the WBSP.
\end{proposition}

\begin{proof}
  Let $(x_n)$ be a weakly null sequence. It is known that in atomic order continuous Banach lattices, the lattice operations are weakly continuous. It follows that $\abs{x_n}\xrightarrow{\rm w}0$. Passing to a subsequence, we may assume by Proposition~\ref{d-wbsp}\eqref{d-wbspi31} that the Ces\`aro means of $\bigl(\abs{x_n}\bigr)$ are almost order bounded. This yields that the Ces\`aro means of $(x_n)$ are almost order bounded. The result now follows from Theorem~\ref{bsps}.
\end{proof}

Alternatively, Proposition~\ref{atomic-DWBSP} follows immediately from Theorem~\ref{d-wbsp}\eqref{d-wbspi2} and the following lemma (cf. \cite[Theorem~1]{Wickstead:77}).

\begin{lemma}
  Every weakly null sequence in an atomic order continuous Banach lattice is uo-null.
\end{lemma}

\begin{proof}
  Suppose not. Then  \cite[Lemma~1.2]{Gao:14} implies $\inf_k\abs{x_{n_k}}>0$ for some subsequence $(x_{n_k})$ of $(x_n)$. There is an atom $a\in X_+$ such that $a<\inf_k\abs{x_{n_k}}$. In particular, $\abs{x_{n_k}}>a$ for every~$k$. Let $f$ be the biorthogonal functional of~$a$, that is, $P_ax=f(x)a$ for every $x\in X$, where $P_a$ is the band projection onto~$B_a$. Then $f$ is a lattice homomorphism, so that
  \begin{math}
    \bigabs{f(x_{n_k})}=f\bigl(\abs{x_{n_k}}\bigr)\ge f(a)=1.
  \end{math}
This contradicts $x_{n_k}\xrightarrow{\rm w}0$.
\end{proof}

Regarding DWBSP$\Rightarrow$DBSP, we have the following result.

\begin{proposition}
A Banach lattice $X$ with the DWBSP has the DBSP if and only if it contains no lattice copy of~$\ell_1$.
\end{proposition}

\begin{proof}
The ``only if'' part follows from the fact that $\ell_1$ fails the DBSP.
For the ``if'' part, suppose that $X$ contains no lattice copies of~$\ell_1$. Then \cite[Theorem~2.4.14]{Meyer-Nieberg:91} guarantees that every norm bounded disjoint sequence in $X$ is weakly null and, therefore, DWBSP yields DBSP.
\end{proof}

\subsection{Banach-Saks operators}

\begin{definition}
An operator $T$ from a Banach space $X$ to a Banach space $Y$ is called a \term{Banach-Saks} (respectively, \term{weakly Banach-Saks}) operator if for any norm bounded (respectively, weakly null) sequence $(x_n)$ in~$X$, $(Tx_n)$ has a Ces\`aro convergent subsequence.
\end{definition}

The following is a useful characterization of (weakly) Banach-Saks operators. The proof of this result is an immediate application of Lemma~\ref{aobdd-BSP}.

\begin{theorem}\label{BSO}
Let $X$ be a Banach space and $Y$ be an order continuous Banach lattice. For an operator $T\colon X\rightarrow Y$, the following are equivalent.
\begin{enumerate}
\item\label{BSOi1} $T$ is a Banach-Saks (respectively,~weakly Banach-Saks) operator,
\item\label{BSOi2} For every norm bounded (respectively,~weakly null) sequence $(x_n)$ in~$X$, there is a subsequence $(x_{n_k})$ such that the Ces\`aro means of any subsequence of $(Tx_{n_k})$ are norm and uo-convergent to some~$y$.
\item\label{BSOi3} For every norm bounded (respectively,~weakly null) sequence $(x_n)$ in~$X$, $(Tx_n)$ has a subsequence whose Ces\`aro means are almost order bounded.
\end{enumerate}
\end{theorem}

This theorem allows us to present a simple proof of the following result, which was originally proved in \cite[Corollary~3.3]{Flores:08} by different methods. Our proof demonstrates the efficiency of the approach of transferring topological properties to order properties.

\begin{corollary}[{\cite{Flores:08}}]\label{bsd}Let $X$ and $Y$ be Banach lattices with $Y$ order continuous. If $0\leq S\leq T:X\rightarrow Y$ with $T$ Banach-Saks, then $S$ is also Banach-Saks.
\end{corollary}

\begin{proof}
Let $(x_n)$ be a norm bounded sequence in~$X$. Clearly, the sequence $\bigl(\abs{x_n}\bigr)$ is also bounded. Then there exists a subsequence such that the Ces\`aro means of $(T\abs{x_{n_k}})$ are convergent and, therefore, almost order bounded. Since
  \begin{equation}\label{Cesaro-domin}
   \Bigabs{\frac1m\sum_{k=1}^mSx_{n_k}}
   \le\frac1m\sum_{k=1}^mS\abs{x_{n_k}}
   \leq\frac1m\sum_{k=1}^mT\abs{x_{n_k}},
  \end{equation}
the Ces\`aro means of $(Sx_{n_k})$ are also almost order bounded. Hence, $S$ is a Banach-Saks operator by Theorem~\ref{BSO}.
\end{proof}

\begin{remark}
  In \cite{Flores:08}, Corollaries~3.2 and~3.3 are deduced from Theorem~3.1.  Having just presented an alternative proof of their Corollaries~3.3, we note that it easily implies Theorem~3.1 of~\cite{Flores:08}. Indeed, consider the second diagram on page~98 of~\cite{Flores:08}. Since $T_1$ is a Banach-Saks operator, so is $\phi T_1$. Since $F$ is order continuous and $0\le\phi R_1\le\phi T_1$, we conclude that $\phi R_1$ is a Banach-Saks operator. It follows that $R_2R_1=Q\phi R_1$ is Banach-Saks.
\end{remark}

The domination problem for weakly Banach-Saks positive operators remains open. We present the following results. Following \cite{Groenewegen:82,Chen:98}, we say that a Banach lattice has the \term{W1 property} if for every relatively weakly compact set~$A$, the set
\begin{math}
  \bigl\{\abs{a}\mid a\in A\bigr\}
\end{math}
is again relatively weakly compact. This class of spaces includes KB-spaces, atomic order continuous Banach lattices, and AM-spaces.

\begin{theorem}\label{dwbs}
Let $X $ and $Y$ be Banach lattices such that $X$ has the W1 property and $Y$ is order continuous. If $0\leq S\leq T:X\rightarrow Y$ with $T$ weakly Banach-Saks, then $S$ is a weakly Banach-Saks operator.
\end{theorem}

\begin{proof}
Let $(x_n)$ be a weakly null sequence in~$X$.
By the property (W1), we may assume, by passing to a subsequence, that
$\abs{x_n}\xrightarrow{\rm w}a$ for some $a\in X$. Since $T$ is weakly Banach-Saks, passing to a further subsequence we may assume by Theorem~\ref{BSO} that the Ces\`aro means of $\bigl(T\abs{x_n}-Ta\bigr)$ and, therefore, of $\bigl(T\abs{x_n}\bigr)$ are almost order bounded. As in~\eqref{Cesaro-domin}, we conclude that the Ces\`aro means of $(Sx_n)$ are almost order bounded. Apply Theorem~\ref{BSO} again.
\end{proof}

\begin{proposition}\label{dwbs2} Let $X $ and $Y$ be Banach lattices such that $X$ is an order continuous Banach lattice with the  subsequence splitting property. If $0\leq S\leq T\colon X\rightarrow Y$ with $T$ weakly Banach-Saks, then $S$ is a weakly Banach-Saks operator.
\end{proposition}

\begin{proof}
Let $(x_n)$ be a weakly null sequence in~$X$. Since $X$ has the subsequence splitting property, we may assume by passing to a subsequence that $x_n=z_n+y_n$, where $(z_n)$ is disjoint and $(y_n)$ is almost order bounded.
Passing to a further subsequence, we may assume that every subsequence of $(y_n)$ and, therefore, of $(Sy_n)$, is Ces\`aro convergent by Lemma~\ref{aobdd-BSP}. Recall from Remark~\ref{SSP-w-null} that $\abs{z_n}\xrightarrow{\rm w}0$. Therefore, after passing to a further subsequence, $\bigl(T\abs{z_n}\bigr)$ is Ces\`aro null. It follows that $(Sz_n)$ is Ces\`aro null. Hence, $(Sx_n)$ is Ces\`aro convergent.
\end{proof}

Note that the proof works whenever $X$ has the subsequence splitting property for weakly null sequences. Cf.~also \cite[Theorem~1.1]{Flores:06}. 

We would like to finish this section with an open problem.

\begin{problem}
  Can one remove or relax the assumptions on $X$ and $Y$ in Theorem~\ref{dwbs} and Proposition~\ref{dwbs2}?
\end{problem}

\bigskip

\textbf{Acknowledgement.} We would like to thank T.~Oikhberg for valuable discussions. Most of the work on this paper was done at the University of Alberta. The authors would like to thank the University of Alberta for hospitality and support.

\end{document}